\documentclass[letter,10pt]{amsart}
\usepackage{amssymb,amsmath,amsfonts}
\usepackage{mathrsfs}
\usepackage[left=1 in, right=1 in,top=1 in, bottom=1 in]{geometry}
\usepackage{mathenv}
\usepackage{framed}

\newtheorem{theorem}{Theorem}[section]
\newtheorem{lemma}[theorem]{Lemma}

\newtheorem{corollary}[theorem]{Corollary}

\newtheorem{remark}[theorem]{Remark}

\makeatletter
\renewcommand \theequation {%
\ifnum \c@section>\z@ \@arabic\c@section.%
\fi\@arabic\c@equation} \@addtoreset{equation}{section}
\makeatother

\providecommand{\ud}[1]{\mathrm{d}{#1}}
\providecommand{\abs}[1]{\left\vert#1\right\vert}
\providecommand{\nm}[1]{\left\Vert#1\right\Vert}

\def\p{\partial}

\def\rt{\rightarrow}
\def\r{\mathbb{R}}
\def\no{\nonumber}

\def\ds{\displaystyle}

\def\u{\mathscr{U}}
\def\w{\mathscr{W}}
\def\c{\mathcal{C}}
\def\rr{\mathcal{R}}
\def\k{\mathcal{K}}
\def\i{{\bf{1}}}
\def\f{\mathfrak{F}}
\def\t{t}

\def\ep{\epsilon}
\def\v{\vec}
\def\na{\nabla}
\def\de{\Delta}

\def\uu{\mathfrak{U}}
\def\ww{\mathfrak{W}}

\def\uk{\mathfrak{u}}
\def\wk{\mathfrak{w}}

\begin{document}
\title{Renormalized Energy for Dislocations in Quasi-Crystals}

\author{Lei Wu}
\thanks{Department of Mathematical Sciences, Carnegie Mellon University, Pittsburgh, PA 15213, USA.}
\thanks{Email: lwu2@andrew.cmu.edu}

%
%

\begin{abstract}
Anti-plane shear deformations of a hexagonal quasi-crystal with multiple screw dislocations are considered. Using a variational formulation, the elastic equilibrium is characterized via limit of minimizers of a core-regularized energy functional. A sharp estimate of the asymptotic energy when the core radius tends to zero is obtained using higher-order $\Gamma$-convergence. Also, the interaction between dislocations and the Peach-K\"{o}hler force at each dislocation are analyzed.
\ \\
\textbf{Keywords:} dislocation; renormalized energy; $\Gamma$-convergence.
\end{abstract}

\maketitle

\pagestyle{myheadings} \thispagestyle{plain} \markboth{LEI WU}{RENORMALIZED ENERGY FOR DISLOCATIONS IN QUASI-CRYSTALS}


%
%

\section{Introduction}

\subsection{Problem Settings}

Quasi-crystals were introduced in 1982 by Shechtman(see \cite{Shechtman.Blech.Gratias.Cahn1984}) as a kind of non-crystalline condensed matter state. In contrast with crystals with periodic atomic arrangement, quasi-crystals only exhibit quasi-periodicity, i.e. they have perfect long-range order (like mirror symmetry) but no three-dimensional periodicity.

Unlike many other amorphous solids, quasi-crystals have similar elastic properties to these of crystals. More importantly, based on the Landau density wave theory(see \cite{Fan2011}),
quasi-crystals can be described as a projection of higher-dimensional crystals into a lower-dimensional space. This requires two displacement fields $\v u$ and $\v w$ defined in the physical domain of the quasi-crystal, where $\v u$ is a phonon field which is similar to the displacement field in crystals and $\v w$ is an extra phase field. Also,
we may define the strain and stress tensors in phonon space and phase space.

To be precise, we consider anti-plane shear deformations of a one-dimensional hexagonal quasi-crystal (see \cite{Blass.Morandotti2014}, \cite{Blass.Fonseca.Leoni.Morandotti2015}, \cite{Cermelli.Gurtin1999}, \cite{Cermelli.Leoni2005}, \cite{Fan2011}, \cite{Liu.Guo.Fan2005}). Given an elastic body $\Xi=\Omega\times\r$, where $\Omega\subset\r^2$ is simply-connected, bounded and open, with Lipschitz $\p\Omega$, we denote the phonon deformation as
\begin{eqnarray*}
\Phi: (x,y,z)\rt (x,y,z+u(x,y)),
\end{eqnarray*}
and the phase deformation as
\begin{eqnarray*}
\Psi: (x,y,z)\rt (x,y,z+w(x,y)),
\end{eqnarray*}
for some functions $u,w: \Omega\rt\r$. This allows us to reduce the three-dimensional problem to a two-dimensional setting.
Hence, the phonon strain tensor is defined as
\begin{eqnarray}\label{it 01}
{U}&:=&\na(0,0,u)=\left(
\begin{array}{ccc}
0&0&0\\\rule{0ex}{2.0em}
0&0&0\\\rule{0ex}{2.0em}
\dfrac{\p u}{\p x}&\dfrac{\p u}{\p y}&0
\end{array}
\right),
\end{eqnarray}
which can be symmetrized as
\begin{eqnarray*}
\tilde{U}&:=&\frac{{U}+{U}^T}{2}=\left(
\begin{array}{ccc}
0&0&\dfrac{1}{2}\dfrac{\p u}{\p x}\\\rule{0ex}{2.0em}
0&0&\dfrac{1}{2}\dfrac{\p u}{\p y}\\\rule{0ex}{2.0em}
\dfrac{1}{2}\dfrac{\p u}{\p x}&\dfrac{1}{2}\dfrac{\p u}{\p y}&0
\end{array}
\right),
\end{eqnarray*}
and the non-symmetric phase strain tensor is defined as
\begin{eqnarray}\label{it 02}
{W}&:=&\na(0,0,w)=\left(
\begin{array}{ccc}
0&0&0\\\rule{0ex}{2.0em}
0&0&0\\\rule{0ex}{2.0em}
\dfrac{\p w}{\p x}&\dfrac{\p w}{\p y}&0
\end{array}
\right).
\end{eqnarray}
The relations (\ref{it 01}) and (\ref{it 02}) hold for a quasi-crystal when dislocations are absent. If dislocations are taken into consideration, then the strain tensor is singular at the site of the dislocations, and in particular it is a line singularity for a screw dislocation.
Dislocations are one-dimensional defects in a crystalline-type material, whose presence may greatly affect the elastic and other properties (see \cite{Hirth.Lothe1982} and \cite{Nabarro1967}). Dislocation lines of quasi-crystals were observed in experiments soon after Shechtman's discover (see \cite{Agiasofitou.Lazar.Kirchner2010}, \cite{Lazar.Agiasofitou2014}, \cite{Levine.Lubensky.Ostlund.Ramaswamy.Steinhardt.Toner1985}, \cite{Liu.Guo.Fan2005}).

In a quasi-crystal undergoing a shear deformation, a screw dislocation may be described by a position $(x,y)\in\Omega$ and a Burger's vector $\vec b=b\vec e_z$. Here $\v e_z$ denotes the unit vector in the $z$ direction and $b$, the Burger's modulus, represents the magnitude of the dislocation.
The presence of dislocation yields a singularity at position $(x,y)$ and thus strain tensors fail to be the gradients of smooth displacement fields, i.e. (\ref{it 01}) and (\ref{it 02}) do not hold any more.

To be precise, consider $N$ dislocations at $\v d_i=(x_i,y_i)$ for $i=1,2,\ldots,N$, with Burger's vector for the phonon field given by $\v b_u^i=b_u^i\v e_z$ and for the phase field given by $\v b_w^i=b_w^i\v e_z$. The strain tensors ${U}$ and ${W}$ now satisfy
\begin{eqnarray*}
(\na\times{U})\cdot\v e_z=\sum_{i=1}^N\v b_u^i\delta_{\v d_i},\qquad
(\na\times{W})\cdot\v e_z=\sum_{i=1}^N\v b_w^i\delta_{\v d_i},
\end{eqnarray*}
which is equivalent to
\begin{eqnarray*}
b_u^i=\int_{\ell_i}{U}\cdot t\ud{s},\qquad
b_w^i=\int_{\ell_i}{W}\cdot t\ud{s},
\end{eqnarray*}
where $\ell_i$ is any counterclockwise loop that surrounds $\v d_i$ and no other dislocation points, $t$ is the tangent of $\ell_i$ and $\ud{s}$ is the line differential. Similarly, we can still define the symmetrized phonon strain tensor $\tilde{U}=\dfrac{{U}+{U}^T}{2}$.

Denote the phonon stress tensor as $\sigma$ and the phase stress tensor as $\rho$, which are $3\times 3$ matrices in principle. For the convenience of computation, we may straighten $\sigma$, $\rho$, $\tilde{U}$ and ${W}$ to column vectors with $9$ components. Then the generalized Hooke's law (see \cite{Fan2011}) reads as
\begin{eqnarray*}
\left(
\begin{array}{c}
\sigma\\
\rho
\end{array}
\right)=\left(
\begin{array}{cc}
\c&\rr\\
\rr^T&\k
\end{array}
\right)
\left(
\begin{array}{c}
\tilde{U}\\
{W}
\end{array}
\right),
\end{eqnarray*}
where $\c$, $\rr$, $\k$ are $9\times9$ matrices such that $\left(
\begin{array}{cc}
\c&\rr\\
\rr^T&\k
\end{array}
\right)$ is positive definite and depends on the species of the quasi-crystal.
The equilibrium equations are
\begin{eqnarray*}
\na\cdot\sigma=0,\qquad
\na\cdot\rho=0,
\end{eqnarray*}
where the divergence is performed row by row. Here we use straightened vectors and matrices interchangeably. The free energy is
\begin{eqnarray*}
J[{U},{W}]:=\int_{\Xi}\f[{U},{W}]\ud{x}\ud{y}\ud{z},
\end{eqnarray*}
where the energy density $\f$ is given by
\begin{eqnarray*}
\f[{U},{W}]:=\frac{1}{2}\left(
\begin{array}{cc}
\tilde{U}^T&{W}^T
\end{array}
\right)\left(
\begin{array}{cc}
\c&\rr\\
\rr^T&\k
\end{array}
\right)
\left(
\begin{array}{c}
\tilde{U}\\
{W}
\end{array}
\right).
\end{eqnarray*}
We intend to study the structure of the energy associated with this system.

\subsection{Problem Simplification}

Since $U$ and $W$ are sparse matrices, we can reduce the $18$-variable problem to a $4$-variable problem (see \cite{Fan2011}). In particular, for $N$ dislocation points at $\v d_i$,  $i=1,2,\ldots,N$, with Burger's vectors for the phonon field given by $\v b_u^i$ and for the phase field given by $\v b_w^i$, it suffices to consider $\u=(\u_x,\u_y)$ and $\w=(\w_x,\w_y)$ satisfying
\begin{eqnarray}\label{problem}
\left\{
\begin{array}{l}
\left(
\begin{array}{c}
\sigma\\
\rho
\end{array}
\right)=\left(
\begin{array}{cc}
\c&\rr\\
\rr^T&\k
\end{array}
\right)
\left(
\begin{array}{c}
\u\\
\w
\end{array}
\right),\\\rule{0ex}{1.5em}
\na\times\u=\displaystyle\sum_{i=1}^Nb_u^i\delta_{\v d_i},\quad
\na\times\w=\displaystyle\sum_{i=1}^Nb_w^i\delta_{\v d_i},\\\rule{0ex}{1.0em}
\na\cdot\sigma=0,\quad
\na\cdot\rho=0,
\end{array}
\right.
\end{eqnarray}
where $\sigma=(\sigma_x,\sigma_y)$, $\rho=(\rho_x,\rho_y)$ are vectors with $2$ components, and $\c$, $\rr$, $\k$ are $2\times2$ matrices, $\c$ and $\k$ are symmetric and positive definite, $\na\cdot\v f:=\dfrac{\p f_x}{\p x}+\dfrac{\p f_y}{\p y}$ and $\na\times\v f:=\dfrac{\p f_y}{\p x}-\dfrac{\p f_x}{\p y}$. Roughly speaking, $\u$ plays the role of $\left(\dfrac{\p u}{\p x},\dfrac{\p u}{\p y}\right)$ and $\w$ plays the role of $\left(\dfrac{\p w}{\p x},\dfrac{\p w}{\p y}\right)$. Here we omit the symmetrization procedure of $\u$ since it can be directly incorporated into Hooke's law, and we do not change the notation for $\sigma$, $\rho$, $\c$, $\rr$, $\k$.
The free energy is
\begin{eqnarray*}
J[\u,\w]:=\int_{\Omega}\f[\u,\w]\ud{x}\ud{y},
\end{eqnarray*}
with density
\begin{eqnarray*}
\f[\u,\w]:=\frac{1}{2}\left(
\begin{array}{cc}
\u^T&\w^T
\end{array}
\right)\left(
\begin{array}{cc}
\c&\rr\\
\rr^T&\k
\end{array}
\right)
\left(
\begin{array}{c}
\u\\
\w
\end{array}
\right).
\end{eqnarray*}
In a hexagonal quasi-crystal (see \cite{Fan2011}), we may further simplify the Hooke's law as
\begin{eqnarray*}
\c=\left(
\begin{array}{cc}
C&0\\
0&C
\end{array}
\right),\quad\rr=\rr^T=\left(
\begin{array}{cc}
R&0\\
0&R
\end{array}
\right),\quad\k=\left(
\begin{array}{cc}
K&0\\
0&K
\end{array}
\right),
\end{eqnarray*}
for some constants $C$, $R$, $K$ with
\begin{eqnarray}\label{assumption}
C,K>0$ and $CK>R^2,
\end{eqnarray}
i.e. the matrix $\left(
\begin{array}{cc}
\c&\rr\\
\rr^T&\k
\end{array}
\right)$ is positive definite. Also, the free energy density reduces to
\begin{eqnarray}\label{it temp 11}
\f[\u,\w]&=&\frac{1}{2}\bigg(C\abs{\u}^2+K\abs{\w}^2+2R(\u\cdot\w)\bigg).
\end{eqnarray}

\subsection{Core Regularization}

It is well-known that in a neighborhood of a dislocation point, the free energy blows up (see \cite{Cermelli.Gurtin1999} and \cite{Cermelli.Leoni2005}). Similar to the techniques in \cite{Blass.Fonseca.Leoni.Morandotti2015} and \cite{Cermelli.Leoni2005} for crystals, we consider a variational formulation by removing a core $B_{\ep}(\v d_i)=\{\v d=(x,y): \abs{\v d-\v d_i}\leq\ep\}$ around each dislocation, and we consider the minimization problem
\begin{eqnarray}\label{problem.}
\min_{(\uu,\ww)\in H_0^{\ep}}\int_{\Omega_{\ep}}\f[\uu,\ww]\ud{x}\ud{y},
\end{eqnarray}
where $\Omega_{\ep}:=\Omega\bigg\backslash \bigg(\displaystyle\bigcup_{i=1}^NB_{\ep}(\v d_i)\bigg)$ and the admissible set is defined by
\begin{eqnarray*}
H_0^{\ep}&=&\bigg\{(\uu,\ww):\uu,\ww\in L^2(\Omega_{\ep}),\ \ \na\times\uu=0,\na\times\ww=0\ \ \text{in}\ \ \Omega_{\ep},\\
&&\int_{\p B_{\ep}(\v d_i)}\uu\cdot\t\ud{s}=b_u^i,\ \ \int_{\p B_{\ep}(\v d_i)}\ww\cdot\t\ud{s}=b_w^i,\ \ i=1,2,\ldots,N\bigg\},\no
\end{eqnarray*}
where $\t$ is the unit tangent vector at $\p B_{\ep}(\v d_i)$. Here $\uu\cdot\t$ and $\ww\cdot\t$ are the tangential traces of $\uu$ and $\ww$, which are well-defined in the $L^2$ curl-free space(see \cite{Cermelli.Gurtin1999} and \cite{Cermelli.Leoni2005}).

Assume that the solution to the above minimization problem admits a unique solution as $(\u_{\ep},\w_{\ep})$. Our goal is to study the behavior of $(\u_{\ep},\w_{\ep})$ and of the free energy
\begin{eqnarray*}
J_{\ep}[\u_{\ep},\w_{\ep}]:=\displaystyle\int_{\Omega_{\ep}}\f[\u_{\ep},\w_{\ep}]\ud{x}\ud{y},
\end{eqnarray*}
as $\ep\rt0$.

\subsection{Main Theorem}

We intend to use $\Gamma$-convergence to analyze the minimizer and energy structure.
Define the functional $J_{\ep}^{(0)}: L^2(\Omega)\times L^2(\Omega)\rt[0,\infty]$ by
\begin{eqnarray*}
J_{\ep}^{(0)}\left[\uu_{\ep},\ww_{\ep}\right]:=\left\{
\begin{array}{ll}
\ds\int_{\Omega_{\ep}}\frac{1}{2}\bigg(C\abs{\uu_{\ep}}^2+K\abs{\ww_{\ep}}^2+2R(\uu_{\ep}\cdot\ww_{\ep})\bigg)\ud{x}\ud{y}\\
\qquad\ \ \text{if}\ \ (\uu_{\ep},\ww_{\ep})=\left(\dfrac{\tilde\uu_{\ep}}{\abs{\ln(\ep)}^{1/2}},\dfrac{\tilde\ww_{\ep}}{\abs{\ln(\ep)}^{1/2}}\right)\ \ \text{for some}\ \ (\tilde\uu_{\ep},\tilde\ww_{\ep})\in H_0^{\ep},\\
\infty\ \ \text{otherwise in}\ \ L^2(\Omega)\times L^2(\Omega).
\end{array}
\right.
\end{eqnarray*}

\begin{theorem}(Compactness)\label{theorem 2.}(see Section 3.1)
Assume that (\ref{assumption}) holds and $(\uu_{\ep},\ww_{\ep})\in L^2(\Omega)\times L^2(\Omega)$ satisfy
\begin{eqnarray*}
\sup_{\ep>0}J_{\ep}^{(0)}[\uu_{\ep},\ww_{\ep}]\leq C_0.
\end{eqnarray*}
Then there exists $v_u,v_w\in H^1(\Omega)$ such that up to the extraction of subsequence (non-relabelled),
\begin{eqnarray*}
\left({\bf{1}}_{\Omega_{\ep}}\uu_{\ep},{\bf{1}}_{\Omega_{\ep}}\ww_{\ep}\right)\rightharpoonup(\na v_u,\na v_w)\ \text{in}\ \text{weak}-L^2\ \text{as}\ \ep\rt0.
\end{eqnarray*}
\end{theorem}

With compactness theorem in hand, we can show the zeroth-order $\Gamma$-convergence.
\begin{theorem}\label{theorem 3}( $0^{th}$-Order $\Gamma$-Convergence )(see Section 3.2)
Assume that (\ref{assumption}) holds. Define the functional $J_0^{(0)}: L^2(\Omega)\times L^2(\Omega)\rt[0,\infty]$ as
\begin{eqnarray*}
J_{0}^{(0)}[\uu,\ww]:=\left\{
\begin{array}{ll}
\ds\int_{\Omega}\frac{1}{2}\bigg(C\abs{\na v_u}^2+K\abs{\na v_w}^2+2R(\na v_u\cdot\na v_w)\bigg)+\sum_{i=1}^N\frac{C(b_u^i)^2+K(b_w^i)^2+2R(b_u^i)(b_w^i)}{4\pi}&\\
\qquad\qquad\qquad\qquad\qquad\qquad\ \ \text{if}\ \ (\uu,\ww)=\left(\na v_u, \na v_w\right)\ \ \text{for}\ \ \text{some}\ \ v_u,v_w\in H^1(\Omega),\\
\infty\qquad\ \ \text{otherwise in}\ \ L^2(\Omega)\times L^2(\Omega).
\end{array}
\right.
\end{eqnarray*}
Then
\begin{enumerate}
\item
For any sequence of pairs $(\uu_{\ep},\ww_{\ep})\in L^2(\Omega)\times L^2(\Omega)$ such that $\left(\uu_{\ep},\ww_{\ep}\right)\rightharpoonup(\uu,\ww)$ in weak-$L^2(\Omega)$, we have $\liminf_{\ep\rt0}J_{\ep}^{(0)}[\uu_{\ep},\ww_{\ep}]\geq J_{0}^{(0)}[\na v_u,\na v_w]$.
\item
There exists a sequence of pairs $(\uu_{\ep},\ww_{\ep})\in L^2(\Omega)\times L^2(\Omega)$ such that $\left(\uu_{\ep},\ww_{\ep}\right)\rightharpoonup(\uu,\ww)$ in weak-$L^2(\Omega)$, we have $\limsup_{\ep\rt0}J_{\ep}^{(0)}[\uu_{\ep},\ww_{\ep}]\leq J_{0}^{(0)}[\na v_u,\na v_w]$,
\end{enumerate}
which means
\begin{eqnarray*}
J_{\ep}^{(0)}[\uu_{\ep},\ww_{\ep}]\rt J_0^{(0)}[\uu,\ww],
\end{eqnarray*}
in the sense of $\Gamma$-convergence in weak-$L^2(\Omega)$
\end{theorem}

$\Gamma$-convergence naturally yields the convergence of minimum of energy functionals.
\begin{corollary}(Core Energy)(see Section 3.2)\label{corollary 1.}
Assume that (\ref{assumption}) holds. We have
\begin{eqnarray*}
\inf_{ \uu, \ww} J_{0}^{(0)}[\uu,\ww]=\sum_{i=1}^N\frac{C(b_u^i)^2+K(b_w^i)^2+2R(b_u^i)(b_w^i)}{4\pi}.
\end{eqnarray*}
Assume $(\u_{\ep}',\w_{\ep}')$ is the minimizer of $J_{\ep}^{(0)}$, then we have
\begin{eqnarray*}
J_{\ep}^{(0)}[\u_{\ep}',\w_{\ep}']=E_0+o(1),
\end{eqnarray*}
where the rescaled leading-order energy
\begin{eqnarray}\label{mt 001}
E_0=\sum_{i=1}^N\frac{C(b_u^i)^2+K(b_w^i)^2+2R(b_u^i)(b_w^i)}{4\pi}.
\end{eqnarray}
\end{corollary}

The zeroth-order $\Gamma$-convergence result tells us the asymptotic behavior of leading-order free energy. However, the rescaling in $J_{\ep}^{(0)}$ suppress $O(1)$ term in the energy. As \cite{Anzellotti.Baldo1993} revealed, more detailed information can be discovered when we get rid of the rescaling and go to first-order  $\Gamma$-convergence.
Define the functional $J_{\ep}^{(1)}: L^2(\Omega)\times L^2(\Omega)\rt[0,\infty]$ as
\begin{eqnarray*}
J_{\ep}^{(1)}\left[\tilde\uu_{\ep},\tilde\ww_{\ep}\right]:=\left\{
\begin{array}{ll}
\ds\int_{\Omega_{\ep}}\frac{1}{2}\bigg(C\abs{\tilde\uu_{\ep}}^2+K\abs{\tilde\ww_{\ep}}^2+2R(\tilde\uu_{\ep}\cdot\tilde\ww_{\ep})\bigg)\ud{x}\ud{y}-\abs{\ln(\ep)}\inf_{ \uu, \ww} J_0^{(0)}[ \uu, \ww]\\
\qquad\qquad\qquad\qquad\qquad\qquad\qquad\qquad\qquad\qquad\qquad\ \ \text{if}\ \ (\tilde\uu_{\ep},\tilde\ww_{\ep})\in H_0^{\ep},\\
\infty\ \ \text{otherwise in}\ \ L^2(\Omega)\times L^2(\Omega).
\end{array}
\right.
\end{eqnarray*}

\begin{theorem}\label{theorem 4}( $1^{st}$-Order $\Gamma$-Convergence )(see Section 3.3)
Assume that (\ref{assumption}) holds. Define the functional $J_0^{(1)}: L^2(\Omega)\times L^2(\Omega)\rt[0,\infty]$ as
\begin{eqnarray*}
J_0^{(1)}[ \tilde\uu, \tilde\ww]:=\left\{
\begin{array}{ll}
E_{\textnormal{self}}+E_{\textnormal{int}}+E_{\textnormal{elastic}}\ \ \text{if}\ \ ( \tilde\uu, \tilde\ww)=\ds\left(\na v_u+\sum_{i=1}^N\u_i, \na v_w+\sum_{i=1}^N\w_i\right),\\
\qquad\qquad\qquad\qquad\qquad\qquad\qquad\qquad\text{for}\ \ \text{some}\ \ (v_u,v_w)\in H^1(\Omega),\\
\infty\qquad\ \ \text{otherwise in}\ \ L^2(\Omega)\times L^2(\Omega),
\end{array}
\right.
\end{eqnarray*}
where
\begin{eqnarray*}
E_{\textnormal{self}}&:=&\sum_{i=1}^N\int_{\Omega\backslash B_r(\v d_i)}\frac{1}{2}\bigg(C\abs{\u_{i}}^2+K\abs{\w_{i}}^2+2R(\u_{i}\cdot\w_{i})\bigg)\ud{x}\ud{y}\\
&&+\sum_{i=1}^N\frac{(C(b_u^i)^2+K(b_w^i)^2+2R(b_u^i)(b_w^i))}{4\pi}\ln(r),\no\\
E_{\textnormal{int}}&:=&\sum_{i=1}^{N-1}\sum_{j=i}^N\int_{\Omega}\bigg(C(\u_i\cdot\u_j)+K(\w_i\cdot\w_j)+R(\u_i\cdot\w_j)+R(\u_j\cdot\w_i)\bigg),\\
E_{\textnormal{elastic}}&:=&J[\na v_u,\na v_w]+\sum_{i=1}^N\int_{\p\Omega}\bigg(v_u(C\u_i+R\w_i)+v_w(K\w_i+R\u_i)\bigg)\cdot n\ud{s}.
\end{eqnarray*}
Then
\begin{enumerate}
\item
For any sequence of pairs $(\tilde\uu_{\ep},\tilde\ww_{\ep})\in L^2(\Omega)\times L^2(\Omega)$ such that $\left(\tilde\uu_{\ep},\tilde\ww_{\ep}\right)\rightharpoonup(\tilde\uu,\tilde\ww)$ in weak-$L^2(\Omega)$, we have $\liminf_{\ep\rt0}J_{\ep}^{(1)}[\tilde\uu_{\ep},\tilde\ww_{\ep}]\geq J_{0}^{(1)}[\na v_u,\na v_w]$.
\item
There exists a sequence of pairs $(\tilde\uu_{\ep},\tilde\ww_{\ep})\in L^2(\Omega)\times L^2(\Omega)$ such that $\left(\tilde\uu_{\ep},\tilde\ww_{\ep}\right)\rightharpoonup(\tilde\uu,\tilde\ww)$ in weak-$L^2(\Omega)$, we have $\limsup_{\ep\rt0}J_{\ep}^{(1)}[\tilde\uu_{\ep},\tilde\ww_{\ep}]\leq J_{0}^{(1)}[\na v_u,\na v_w]$,
\end{enumerate}
which means
\begin{eqnarray*}
J_{\ep}^{(1)}[\tilde\uu_{\ep},\tilde\ww_{\ep}]\rt J_0^{(1)}[\tilde\uu,\tilde\ww],
\end{eqnarray*}
in the sense of $\Gamma$-convergence in weak-$L^2(\Omega)$.
\end{theorem}

Similarly, we have a better approximation of energy functionals.
\begin{corollary}(Renormalized Energy)(see Section 3.3)\label{corollary 2.}
Assume that (\ref{assumption}) holds. We have
\begin{eqnarray*}
\inf_{ \tilde\uu, \tilde\ww} J_{0}^{(1)}[\tilde\uu,\tilde\ww]=F_{\textnormal{self}}+F_{\textnormal{int}}+F_{\textnormal{elastic}},
\end{eqnarray*}
where
\begin{eqnarray}\label{mt 002}
F_{\textnormal{self}}:&=&\sum_{i=1}^N\int_{\Omega\backslash B_r(\v d_i)}\frac{1}{2}\bigg(C\abs{\u_{i}}^2+K\abs{\w_{i}}^2+2R(\u_{i}\cdot\w_{i})\bigg)\ud{x}\ud{y}\\
&&+\sum_{i=1}^N\frac{(C(b_u^i)^2+K(b_w^i)^2+2R(b_u^i)(b_w^i))}{4\pi}\ln(r),\no\\
F_{\textnormal{int}}:&=&\sum_{i=1}^{N-1}\sum_{j=i}^N\int_{\Omega}\bigg(C(\u_i\cdot\u_j)+K(\w_i\cdot\w_j)+R(\u_i\cdot\w_j)+R(\u_j\cdot\w_i)\bigg),\no\\
F_{\textnormal{elastic}}:&=&J[\na u_{0},\na w_{0}]+\sum_{i=1}^N\int_{\p\Omega}\bigg(u_{0}(C\u_i+R\w_i)+w_{0}(K\w_i+R\u_i)\bigg)\cdot n\ud{s},\no
\end{eqnarray}
in which $(u_0,w_0)$ is the minimizer of
\begin{eqnarray*}
I[v_u,v_w]=J[\na v_u,\na v_w]+\sum_{i=1}^N\int_{\p\Omega}\bigg(v_u(C\u_i+R\w_i)+v_w(K\w_i+R\u_i)\bigg)\cdot n\ud{s}.
\end{eqnarray*}
Assume $(\tilde\u_{\ep}',\tilde\w_{\ep}')\in H_0^{\ep}$ is the minimizer of $J_{\ep}^{(1)}$, then we have
\begin{eqnarray*}
J_{\ep}^{(1)}[\tilde\u_{\ep}',\tilde\w_{\ep}']=F_{\textnormal{self}}+F_{\textnormal{int}}+F_{\textnormal{elastic}}+o(1).
\end{eqnarray*}
\end{corollary}

As corollaries, we can now state a characterization of the structure of minimizer $(\u_{\ep},\w_{\ep})$ and energy $J_{\ep}[\u_{\ep},\w_{\ep}]$ in (\ref{problem.}).
\begin{theorem}\label{theorem 1}(Minimizer Structure)(see Section 3.4)
Assume that (\ref{assumption}) holds. The problem (\ref{problem.}) admits a unique solution
\begin{eqnarray*}
\u_{\ep}=\sum_{i=1}^N\u_i+\na u_{\ep},\qquad \w_{\ep}=\sum_{i=1}^N\w_i+\na w_{\ep},
\end{eqnarray*}
where
\begin{eqnarray*}
\u_i&=&\frac{b_u^i}{2\pi}\frac{1}{(x-x_i)^2+(y-y_i)^2}\bigg(-(y-y_i),(x-x_i)\bigg),\\
\w_i&=&\frac{b_w^i}{2\pi}\frac{1}{(x-x_i)^2+(y-y_i)^2}\bigg(-(y-y_i),(x-x_i)\bigg),
\end{eqnarray*}
and $(u_{\ep},w_{\ep})$ is the unique minimizer of
\begin{eqnarray*}
I_{\ep}[u_{\ep},w_{\ep}]:&=&J_{\ep}[\na u_{\ep},\na w_{\ep}]+\sum_{i=1}^N\int_{\p\Omega}\bigg(u_{\ep}(C\u_i+R\w_i)+w_{\ep}(K\w_i+R\u_i)\bigg)\cdot n\ud{s}\no\\
&&-\sum_{i=1}^N\sum_{j\neq i}\int_{\p B_{\ep}(x_i,y_i)}\bigg(u_{\ep}(C\u_j+R\w_j)+w_{\ep}(K\w_j+R\u_j)\bigg)\cdot n\ud{s},\no
\end{eqnarray*}
subject to
$\displaystyle\int_{B}u_{\ep}\ud{x}\ud{y}=0$ and $\displaystyle\int_{B}w_{\ep}\ud{x}\ud{y}=0$
for some ball $B\subset\Omega_{\ep}$, with $n$ the outward unit normal vector on $\p\Omega$.

Furthermore, $(\u_{\ep},\w_{\ep})$ converges in weak-$L^2(\Omega)$ as $\ep\rt0$ to $(\u_0,\w_0)$ where
\begin{eqnarray*}
\u_{0}=\sum_{i=1}^N\u_i+\na u_{0},\qquad \w_{0}=\sum_{i=1}^N\w_i+\na w_{0}.
\end{eqnarray*}
and $[u_0,w_0]$ is the unique minimizer of
\begin{eqnarray*}
I_{0}[u_0,w_0]&=&J[\na u_{0},\na w_{0}]+\sum_{i=1}^N\int_{\p\Omega}\bigg(u_{0}(C\u_i+R\w_i)+w_{0}(K\w_i+R\u_i)\bigg)\cdot n\ud{s},
\end{eqnarray*}
subject to
$\displaystyle\int_{B}u_0\ud{x}\ud{y}=0$ and $\displaystyle\int_{B}w_0\ud{x}\ud{y}=0$
for some ball $B\subset\Omega_{\ep}$
\end{theorem}

\begin{theorem}\label{theorem 2}(Energy Structure)(see Section 3.4)
Assume that (\ref{assumption}) holds. We have
\begin{eqnarray*}
J_{\ep}[\u_{\ep},\w_{\ep}]=\int_{\Omega_{\ep}}\f[\u_{\ep},\w_{\ep}]\ud{x}\ud{y}&=&E_0\ln\bigg(\frac{1}{\ep}\bigg)+F+o(1),
\end{eqnarray*}
where the core energy $E_0$ is defined in (\ref{mt 001})
and the renormalized energy $F=F_{\textnormal{self}}+F_{\textnormal{int}}+F_{\textnormal{elastic}}$ is defined in (\ref{mt 002}).
\end{theorem}

\begin{remark}
The core energy is a leading singular term of $O(\abs{\ln(\ep)})$, which confirms that the free energy is not finite when dislocations are present. The $O(1)$ term $F$ is usually called the renormalized energy and is physically meaningful. This type of asymptotic expansion was first derived for Ginzburg-Landau vortices in \cite{Bethuel.Brezis.Helein1993}, and extended to the context of dislocation in \cite{Cermelli.Leoni2005}. The techniques to prove $\Gamma$-convergence results were first introduced in the study of the Ginzburg-Landau vortices (see \cite{Alicandroa.Ponsiglione2014} and \cite{Sandier.Serfaty2003}).

Note that the renormalized energy is independent of the radius $\ep$ and thus fully characterizes the energy structure around dislocations.
\end{remark}

As an application of the energy structure, we prove that the interaction energy $F_{\textnormal{int}}$ obeys the inverse logarithmical law of the distance between two dislocations.
\begin{theorem}\label{theorem 5}(Interaction Energy)(see Section 4.1)
Assume that (\ref{assumption}) holds. We have
\begin{eqnarray*}
F_{\textnormal{int}}&=&\sum_{i=1}^{N-1}\sum_{j=i}^N\frac{Cb_u^ib_u^j+Kb_w^ib_w^j+Rb_u^ib_w^j+Rb_w^ib_u^j}{2\pi}\ln\left(\frac{1}{\abs{\v d_i-\v d_j}}\right)+O(1).
\end{eqnarray*}
\end{theorem}

When multiple dislocations are present, defects interact with themselves by means of the so-called Peach-K\"{o}hler force, which is defined as the negative gradient of renormalized energy $F$ at the dislocation points (see \cite{Gurtin1995}).
\begin{theorem}\label{theorem 6}(Peach-K\"{o}hler force)(see Section 4.2)
Assume that (\ref{assumption}) holds. The Peach-K\"{o}hler force acting at $\v d_k$ is given by
\begin{eqnarray*}
\na_{\v d_k}F
&=&-\int_{\p B_r(\v d_k)}\bigg(\f[\u_0,\w_0]\i-(C\u_0\otimes\u_0+K\w_0\otimes\w_0+R\u_0\otimes\w_0+R\w_0\otimes\u_0)\bigg)\cdot n\ud{s},
\end{eqnarray*}
for $r<\dfrac{1}{2}\min_k\Big(dist(\v d_k,\p\Omega)\Big)$.
\end{theorem}
\begin{remark}
The integrand in Theorem \ref{theorem 6}
\begin{eqnarray*}
E=-\bigg(\f[\u_0,\w_0]\i-(C\u_0\otimes\u_0+K\w_0\otimes\w_0+R\u_0\otimes\w_0+R\w_0\otimes\u_0)\bigg).
\end{eqnarray*}
is usually called the Eshelby stress tensor.
\end{remark}

Our paper is organized as follows: in Section 2 we present some preliminary results on the minimization problem (\ref{problem.}) of $J_{\ep}$ for fixed $\ep$; in Section 3 we derive the zeroth-order and first-order $\Gamma$-convergence of the free energy when $\ep\rt0$ and study the structure of minimizer and energy; Finally, in Section 4 we introduce two applications of the renormalized energy: the interaction between dislocations and the Peach-K\"{o}hler force.

\section{Preliminaries}

In this section, we consider the minimization problem (\ref{problem.}) of $J_{\ep}$ for fixed $\ep$.

\subsection{Euler-Lagrange Equation}

We start with the equations that minimizer of $J_{\ep}$ should satisfy and the uniqueness of minimizer.
\begin{lemma}\label{lemma 1}
Assume that (\ref{assumption}) holds and $(\u_{\ep},\w_{\ep})$ is the minimizer of $J_{\ep}$ in $H_0^{\ep}(\Omega)$. Then it satisfies the equations
\begin{eqnarray}\label{temp 1}
\left\{
\begin{array}{rcl}
\nabla\cdot(\c\u_{\ep}+\rr\w_{\ep})=\nabla\cdot(\k\w_{\ep}+\rr^T\u_{\ep})&=&0\ \ \text{in}\ \ \Omega_{\ep},\\
(\c\u_{\ep}+\rr\w_{\ep})\cdot n=(\k\w_{\ep}+\rr^T\u_{\ep})\cdot n&=&0\ \ \text{on}\ \ \p\Omega_{\ep},
\end{array}
\right.
\end{eqnarray}
where $n$ is the outward normal vector to $\p\Omega_{\ep}$. Moreover, the solution to (\ref{temp 1}) is unique.
\end{lemma}
\begin{proof}
The free energy density in $\Omega_{\ep}$ is given by
\begin{eqnarray*}
\f[\uu,\ww]&=&\frac{1}{2}\left(
\begin{array}{cc}
\uu^T&\ww^T
\end{array}
\right)\left(
\begin{array}{cc}
\c&\rr\\
\rr^T&\k
\end{array}
\right)
\left(
\begin{array}{c}
\uu\\
\ww
\end{array}
\right)\\
&=&\frac{1}{2}\bigg(\uu^T\c\uu+\ww^T\k\ww+\uu^T\rr\ww+\ww^T\rr^T\uu\bigg).\no
\end{eqnarray*}
For any $(\uu,\ww)$ and $(\bar\uu,\bar\ww)$ in $H_0^{\ep}$, we must have $\uu-\bar\uu=\na P$ and $\ww-\bar\ww=\na Q$ for some $P,Q\in H^1(\Omega_{\ep})$ due to curl-free condition. Hence, the first-order variation is
\begin{eqnarray*}
\delta J_{\ep}[\uu,\ww](p,q)&=&\lim_{\theta\rt0}\frac{J_{\ep}[\uu+\theta\na p,\ww+\theta\na q]-J_{\ep}[\uu,\ww]}{\theta}\\
&=&-\int_{\Omega_{\ep}}\bigg(p\na\cdot(\c\uu+\rr\ww)+q\na\cdot(\k\ww+\rr^T\uu)\bigg)\ud{x}\ud{y}\no\\
&&+\int_{\p\Omega_{\ep}}\bigg(p(\c\uu+\rr\ww)\cdot n+q(\k\ww+\rr^T\uu)\cdot n\bigg)\ud{s}.\no
\end{eqnarray*}
Thus, setting $\delta J_{\ep}[\uu,\ww](p,q)=0$ for any $p,q\in H^1(\Omega_{\ep})$, we can deduce that the minimizer $(\u_{\ep},\w_{\ep})$ is a weak solution of the Euler-Lagrange equations (\ref{temp 1}).

To prove uniqueness, assume that $(\u_{\ep},\w_{\ep})$ and $(\bar\u_{\ep},\bar\w_{\ep})$ are two solutions to (\ref{temp 1}). The difference $(f,g)=(\u_{\ep}-\bar\u_{\ep},\w_{\ep}-\bar\w_{\ep})$ must be curl-free and has zero loop integral around $\p B_{\ep}(\v d_i)$. Therefore, we must have $(f,g)=(\na F,\na G)$ for some $F,G\in H^1(\Omega_{\ep})$. Since $F$ and $G$ satisfy the equation
\begin{eqnarray*}
\int_{\Omega_{\ep}}\bigg((\na p)^T(\c\na F+\rr\na G)+(\na q)^T(\k\na G+\rr^T\na F)\bigg)\ud{x}\ud{y}=0,
\end{eqnarray*}
for any $p,q\in H^1(\Omega_{\ep})$, taking $p=F$ and $q=G$, considering $\left(
\begin{array}{cc}
\c&\rr\\
\rr^T&\k
\end{array}
\right)$ is positive definite, we must have $\na F=\na G=0$, and the uniqueness follows.
\end{proof}

\subsection{Estimate and Energy for Single Dislocation}

In this section, we further restrict the discussion to the case in which $\Omega=B_r(\v d_0)$ for constant $r>>\ep$, with only one dislocation at $\v d_0=(x_0,y_0)$ with Burger's vector of phonon field as $\v b_u$ and of phase field as $\v b_w$. Solving the above Euler-Lagrange equations (\ref{temp 1}), by a linear combination, we get
\begin{eqnarray*}
\left\{
\begin{array}{rcl}
\na\cdot\u_{\ep}=\na\cdot\w_{\ep}&=&0\ \ \text{in}\ \ \Omega_{\ep},\\
\u_{\ep}\cdot n=\w_{\ep}\cdot n&=&0\ \ \text{on}\ \ \p\Omega_{\ep},
\end{array}
\right.
\end{eqnarray*}
in $H_0^{\ep}(\Omega)$.
Hence, there exists potential functions $U_{\ep}(x,y)$ and $W_{\ep}(x,y)$ such that $\na U_{\ep}=\u_{\ep}$, $\na W_{\ep}=\w_{\ep}$ and
\begin{eqnarray*}
\de U_{\ep}=\de W_{\ep}=0\ \ \text{in}\ \ \Omega_{\ep}.
\end{eqnarray*}
Therefore, we are lead to solving Laplace's equations in an annulus with Neumann boundary $\dfrac{\p U_{\ep}}{\p n}=\dfrac{\p W_{\ep}}{\p n}=0$. This system has a unique solution subject to the normalization conditions $\displaystyle\int_{\p B_{\ep}(\v d_0)}\ud{U_{\ep}}=b_u$ and $\displaystyle\int_{\p B_{\ep}(\v d_0)}\ud{W_{\ep}}=b_w$, and we obtain the explicit solution as
\begin{eqnarray*}
U_{\ep}=\frac{b_u}{2\pi}\arctan\bigg(\frac{y-y_0}{x-x_0}\bigg),\quad W_{\ep}=\frac{b_w}{2\pi}\arctan\bigg(\frac{y-y_0}{x-x_0}\bigg)\ \ \text{for}\ \ (x,y)\in\Omega_{\ep}.
\end{eqnarray*}
Hence, we have
\begin{eqnarray}
\u_{\ep}&=&\frac{b_u}{2\pi}\frac{1}{(x-x_0)^2+(y-y_0)^2}\bigg(-(y-y_0),(x-x_0)\bigg),\label{mt temp 12}\\
\w_{\ep}&=&\frac{b_w}{2\pi}\frac{1}{(x-x_0)^2+(y-y_0)^2}\bigg(-(y-y_0),(x-x_0)\bigg),\label{mt temp 13}
\end{eqnarray}
for $(x,y)\in\Omega_{\ep}$, and we note that these are independent of $\ep$ and $r$. Therefore, the minimum free energy can be obtained explicitly as
\begin{eqnarray*}
J_{\ep}=\int_{\Omega_{\ep}}\f[\u_{\ep},\w_{\ep}]\ud{x}\ud{y}=(Cb_u^2+Kb_w^2+2Rb_ub_w)\frac{1}{4\pi}\ln\bigg(\frac{r}{\ep}\bigg).
\end{eqnarray*}

\subsection{Estimate and Energy for Multiple Dislocations}

Now we consider the case with multiple dislocations in general domains. For fixed $\v d_i=(x_i,y_i)$, assume that the single-dislocation solution is $(\u_i,\w_i)$. Based on analysis in Lemma \ref{lemma 1}, we must have
\begin{eqnarray*}
\u_{\ep}:=\sum_{i=1}^N\u_i+\na u_{\ep},\qquad \w_{\ep}:=\sum_{i=1}^N\w_i+\na w_{\ep}.
\end{eqnarray*}
for some $u_{\ep},w_{\ep}\in H^1(\Omega_{\ep})$. We deduce
\begin{eqnarray*}
J_{\ep}[\u_{\ep},\w_{\ep}]&=&I_{\ep}[u_{\ep},w_{\ep}]+\sum_{i=1}^NJ_{\ep}[\u_i,\w_i]\\
&&+\sum_{i=1}^{N-1}\sum_{j=i}^N\int_{\Omega_{\ep}}\bigg(C(\u_i\cdot\u_j)+K(\w_i\cdot\w_j)+R(\u_i\cdot\w_j)+R(\u_j\cdot\w_i)\bigg),\no
\end{eqnarray*}
where
\begin{eqnarray*}
I_{\ep}[u_{\ep},w_{\ep}]:&=&J_{\ep}[\na u_{\ep},\na w_{\ep}]+\sum_{i=1}^N\int_{\p\Omega}\bigg(u_{\ep}(C\u_i+R\w_i)+w_{\ep}(K\w_i+R\u_i)\bigg)\cdot n\ud{s}\no\\
&&-\sum_{i=1}^N\sum_{j\neq i}\int_{\p B_{\ep}(x_i,y_i)}\bigg(u_{\ep}(C\u_j+R\w_j)+w_{\ep}(K\w_j+R\u_j)\bigg)\cdot n\ud{s}.\no
\end{eqnarray*}
Therefore, in order to minimize $J_{\ep}$, it suffices to consider the problem:\\
\ \\
($M_{\ep}$): Minimize $I_{\ep}[\uk,\wk]$ for $\uk,\wk\in H^1(\Omega_{\ep})$ subject to
$\displaystyle\int_{B}\uk\ud{x}\ud{y}=0$ and $\displaystyle\int_{B}\wk\ud{x}\ud{y}=0$
for some ball $B\subset\Omega_{\ep}$, i.e. find the solution of
\begin{eqnarray}\label{mini problem}
\min_{\uk,\wk\in H^1(\Omega_{\ep})}I_{\ep}[\uk,\wk].
\end{eqnarray}

This normalization is for the convenience of coercivity and will not affect the minimizing process since adding a constant to $\uk$ or $\wk$ will not affect the value of $I_{\ep}[\uk,\wk]$.

\begin{lemma}
Assume that (\ref{assumption}) holds and $(u_{\ep},w_{\ep})$ is the solution of the minimization problem (\ref{mini problem}) for $I_{\ep}$. Then it satisfies the equations
\begin{eqnarray}\label{temp 2}
\left\{
\begin{array}{rcl}
\nabla\cdot(\c\na u_{\ep}+\rr\na w_{\ep})=\nabla\cdot(\k\na w_{\ep}+\rr^T\na u_{\ep})&=&0\ \ \text{in}\ \ \Omega_{\ep},\\
\Bigg(\c\bigg(\ds\sum_{k=1}^N\u_i+\na u_{\ep}\bigg)+\rr\bigg(\ds\sum_{k=1}^N\w_i+\na w_{\ep}\bigg)\Bigg)\cdot n&=&0\ \ \text{on}\ \ \p\Omega,\\
\Bigg(\k\bigg(\ds\sum_{k=1}^N\w_i+\na w_{\ep}\bigg)+\rr^T\bigg(\ds\sum_{k=1}^N\u_i+\na u_{\ep}\bigg)\Bigg)\cdot n&=&0\ \ \text{on}\ \ \p\Omega,\\
\Bigg(\c\bigg(\ds\sum_{j\neq i}\u_i+\na u_{\ep}\bigg)+\rr\bigg(\ds\sum_{j\neq i}\w_i+\na w_{\ep}\bigg)\Bigg)\cdot n&=&0\ \ \text{on}\ \ \p B_{\ep}(\v d_i),\\
\Bigg(\k\bigg(\ds\sum_{j\neq i}\w_i+\na w_{\ep}\bigg)+\rr^T\bigg(\ds\sum_{j\neq i}\u_i+\na u_{\ep}\bigg)\Bigg)\cdot n&=&0\ \ \text{on}\ \ \p B_{\ep}(\v d_i).
\end{array}
\right.
\end{eqnarray}
Moreover, the solution to (\ref{temp 2}) is unique.
\end{lemma}
\begin{proof}
This follows a standard argument via first-order variation. Letting
\begin{eqnarray*}
\delta I_{\ep}[\uk,\wk](p,q)&=&\lim_{\theta\rt0}\frac{I_{\ep}[\uk+\theta p,\wk+\theta q]-I_{\ep}[\uk,\wk]}{\theta}\\
&=&-\int_{\Omega_{\ep}}\bigg(p\na\cdot(\c\na \uk+\rr\na \wk)+q\na\cdot(\k\na \wk+\rr^T\na \uk)\bigg)\ud{x}\ud{y}\no\\
&&+\int_{\p\Omega}p\Bigg(\c\bigg(\ds\sum_{k=1}^N\u_i+\na \uk\bigg)+\rr\bigg(\ds\sum_{k=1}^N\w_i+\na \wk\bigg)\Bigg)\cdot n\ud{s}\no\\
&&+\int_{\p\Omega}q\Bigg(\k\bigg(\ds\sum_{k=1}^N\w_i+\na \wk\bigg)+\rr^T\bigg(\ds\sum_{k=1}^N\u_i+\na \uk\bigg)\Bigg)\cdot n\ud{s}\no\\
&&-\int_{\p B_{\ep}(\v d_i)}p\Bigg(\c\bigg(\ds\sum_{j\neq i}\u_i+\na \uk\bigg)+\rr\bigg(\ds\sum_{j\neq i}\w_i+\na \wk\bigg)\Bigg)\cdot n\ud{s}\no\\
&&-\int_{\p B_{\ep}(\v d_i)}q\Bigg(\k\bigg(\ds\sum_{j\neq i}\w_i+\na \wk\bigg)+\rr^T\bigg(\ds\sum_{j\neq i}\u_i+\na \uk\bigg)\Bigg)\cdot n\ud{s}.\no
\end{eqnarray*}
If $\delta I_{\ep}[\uk,\wk](p,q)=0$ for any $p,q\in H^1(\Omega_{\ep})$, then the system (\ref{temp 2}) is satisfied. The uniqueness follows from a standard argument as in the proof of Lemma \ref{lemma 1}.
\end{proof}

\subsection{Minimization of the Energy}

\begin{lemma}\label{mini lemma}
Assume that (\ref{assumption}) holds. There exist constants $C_1,C_2>0$ independent of $\ep$ such that
\begin{eqnarray*}
I_{\ep}[\uk,\wk]\geq C_1\bigg(\nm{\uk}_{H^1(\Omega_{\ep})}^2+\nm{\wk}_{H^1(\Omega_{\ep})}^2\bigg)-C_2\bigg(\nm{\uk}_{H^1(\Omega_{\ep})}+\nm{\wk}_{H^1(\Omega_{\ep})}\bigg),
\end{eqnarray*}
for all $\uk,\wk\in H^1(\Omega_{\ep})$ subject to the normalization condition $\displaystyle\int_{B}\uk\ud{x}\ud{y}=0$ and $\displaystyle\int_{B}\wk\ud{x}\ud{y}=0$ for some ball $B\subset\Omega_{\ep}$. Moreover, the minimization problem (\ref{mini problem}) for $I_{\ep}$ admits a unique solution $(u_{\ep},w_{\ep})\in H^1(\Omega_{\ep})$ satisfying
\begin{eqnarray*}
\nm{\uk}_{H^1(\Omega_{\ep})}^2+\nm{\wk}_{H^1(\Omega_{\ep})}^2\leq M,
\end{eqnarray*}
for some constant $M>0$ independent of $\ep$.
\end{lemma}
\begin{proof}
Recall that
\begin{eqnarray*}
I_{\ep}[\uk,\wk]:&=&J_{\ep}[\na \uk,\na \wk]+\sum_{i=1}^N\int_{\p\Omega}\bigg(\uk(C\u_i+R\w_i)+\wk(K\w_i+R\u_i)\bigg)\cdot n\ud{s}\no\\
&&-\sum_{i=1}^N\sum_{j\neq i}\int_{\p B_{\ep}(x_i,y_i)}\bigg(\uk(C\u_j+R\w_j)+\wk(K\w_j+R\u_j)\bigg)\cdot n\ud{s}.\no
\end{eqnarray*}
Since $\f$ is positive definite, we directly estimate
\begin{eqnarray*}
I_{\ep}[\uk,\wk]&\geq&C\int_{\Omega_{\ep}}\bigg(\abs{\na \uk}^2+\abs{\na \wk}^2\bigg)\ud{x}\ud{y}-C'\int_{\p\Omega}\bigg(\abs{\uk}+\abs{\wk}\bigg)\ud{s}-C'\int_{\p B_{\ep}(x_i,y_i)}\bigg(\abs{\uk}+\abs{\wk}\bigg)\ud{s}.\no
\end{eqnarray*}
By Poincar\'{e}'s inequality (see \cite{Cermelli.Leoni2005}), we have for $C_1>0$ independent of $\ep$,
\begin{eqnarray*}
\int_{\Omega_{\ep}}\abs{\na \uk}^2\ud{x}\ud{y}\geq C_1\nm{\uk}_{H^1(\Omega_{\ep})}^2,\\
\int_{\Omega_{\ep}}\abs{\na \wk}^2\ud{x}\ud{y}\geq C_1\nm{\wk}_{H^1(\Omega_{\ep})}^2.
\end{eqnarray*}
In these two estimates, the normalization condition is essential. Also, we have for $C_2>0$ independent of $\ep$,
\begin{eqnarray}
\int_{\p\Omega}\bigg(\abs{\uk}+\abs{\wk}\bigg)\ud{s}&\leq& C_2\bigg(\nm{\uk}_{H^1(\Omega_{\ep})}+\nm{\wk}_{H^1(\Omega_{\ep})}\bigg),\label{tt 1}\\
\int_{\p B_{\ep}(x_i,y_i)}\bigg(\abs{\uk}+\abs{\wk}\bigg)\ud{s}&\leq& C_2\bigg(\nm{\uk}_{H^1(\Omega_{\ep})}+\nm{\wk}_{H^1(\Omega_{\ep})}\bigg)\label{tt 2}.
\end{eqnarray}
Hence, the coercivity is naturally valid, i.e.
\begin{eqnarray*}
I_{\ep}[\uk,\wk]\geq C_1\bigg(\nm{\uk}_{H^1(\Omega_{\ep})}^2+\nm{\wk}_{H^1(\Omega_{\ep})}^2\bigg)-C_2\bigg(\nm{\uk}_{H^1(\Omega_{\ep})}+\nm{\wk}_{H^1(\Omega_{\ep})}\bigg).
\end{eqnarray*}
Since $I_{\ep}$ is strictly convex (see \cite{Cermelli.Leoni2005}) and $I_{\ep}[0,0]=0$, the existence and uniqueness follow.
\end{proof}

We have established the following result.
\begin{theorem}\label{prelim}
Assume that (\ref{assumption}) holds. The problem (\ref{problem.}) admits a unique solution
\begin{eqnarray*}
\u_{\ep}=\sum_{i=1}^N\u_i+\na u_{\ep},\qquad \w_{\ep}=\sum_{i=1}^N\w_i+\na w_{\ep},
\end{eqnarray*}
where
\begin{eqnarray*}
\u_i&=&\frac{b_u^i}{2\pi}\frac{1}{(x-x_i)^2+(y-y_i)^2}\bigg(-(y-y_i),(x-x_i)\bigg),\\
\w_i&=&\frac{b_w^i}{2\pi}\frac{1}{(x-x_i)^2+(y-y_i)^2}\bigg(-(y-y_i),(x-x_i)\bigg),
\end{eqnarray*}
and $(u_{\ep},w_{\ep})$ is the minimizer of
\begin{eqnarray*}
I_{\ep}[u_{\ep},w_{\ep}]:&=&J_{\ep}[\na u_{\ep},\na w_{\ep}]+\sum_{i=1}^N\int_{\p\Omega}\bigg(u_{\ep}(C\u_i+R\w_i)+w_{\ep}(K\w_i+R\u_i)\bigg)\cdot n\ud{s}\no\\
&&-\sum_{i=1}^N\sum_{j\neq i}\int_{\p B_{\ep}(x_i,y_i)}\bigg(u_{\ep}(C\u_j+R\w_j)+w_{\ep}(K\w_j+R\u_j)\bigg)\cdot n\ud{s}.\no
\end{eqnarray*}
subject to
$\displaystyle\int_{B}u_{\ep}\ud{x}\ud{y}=0$ and $\displaystyle\int_{B}w_{\ep}\ud{x}\ud{y}=0$
for some ball $B\subset\Omega_{\ep}$, with $n$ the outward unit normal vector on $\p\Omega$.
\end{theorem}

This theorems tells us the existence and uniqueness of minimizer in (\ref{problem.}). The asymptotic behaviors of minimizer and energy as $\ep\rt0$ are left open at this stage.

\section{$\Gamma$-Convergence}

In this section, we use higher-order $\Gamma$-convergence to dig more information into the structure of minimizer and energy.

\subsection{Weak-$L^2$ Compactness}

Notice that for any $(\tilde\uu_{\ep},\tilde\ww_{\ep})\in H_0^{\ep}$, using (\ref{assumption}), we have
\begin{eqnarray*}
&&\int_{\Omega_{\ep}}\bigg(C\abs{\tilde\uu_{\ep}}^2+K\abs{\tilde\ww_{\ep}}^2+2R(\tilde\uu_{\ep}\cdot\tilde\ww_{\ep})\bigg)\ud{x}\ud{y}\\
&\geq&\sum_{i=1}^N\int_{B_r(\v d_i)\backslash B_{\ep}(\v d_i)}\bigg(C\abs{\tilde\uu_{\ep}}^2+K\abs{\tilde\ww_{\ep}}^2+2R(\tilde\uu_{\ep}\cdot\tilde\ww_{\ep})\bigg)\ud{x}\ud{y}\no\\
&\geq&C_0\sum_{i=1}^N\int_{B_r(\v d_i)\backslash B_{\ep}(\v d_i)}\bigg(\abs{\tilde\uu_{\ep}}^2+\abs{\tilde\ww_{\ep}}^2\bigg)\ud{x}\ud{y}=C_0\sum_{i=1}^N\int_{\ep}^r\int_{B_{\rho}(\v d_i)}\bigg(\abs{\tilde\uu_{\ep}}^2+\abs{\tilde\ww_{\ep}}^2\bigg)\ud{s}\ud{\rho}\no\\
&\geq&C_0\sum_{i=1}^N\int_{\ep}^r\frac{1}{2\pi\rho}\bigg(\int_{B_{\rho}(\v d_i)}(\tilde\uu_{\ep}\cdot\t+\tilde\ww_{\ep}\cdot\t)\ud{s}\bigg)^2\ud{\rho}=C_0\sum_{i=1}^N\int_{\ep}^r\frac{1}{2\pi\rho}(b_u^i+b_w^i)^2\ud{\rho}\no\\
&=&C_0\sum_{i=1}^N\frac{(b_u^i+b_w^i)^2}{2\pi}\ln\left(\frac{r}{\ep}\right).\no
\end{eqnarray*}
Therefore, we know the energy blows up when $\ep\rt0$. We need a proper scaling in order to show compactness. For the minimizer $(\u_{\ep},\w_{\ep})$, we may directly estimate
\begin{eqnarray}\label{et temp 11}
&&\int_{\Omega_{\ep}}\bigg(C\abs{\u_{\ep}}^2+K\abs{\w_{\ep}}^2+2R(\u_{\ep}\cdot\w_{\ep})\bigg)\ud{x}\ud{y}\\
&=&\sum_{i=1}^N\int_{B_r(\v d_i)\backslash B_{\ep}(\v d_i)}\bigg(C\abs{\u_{\ep}}^2+K\abs{\w_{\ep}}^2+2R(\u_{\ep}\cdot\w_{\ep})\bigg)\ud{x}\ud{y}\no\\
&&+\int_{\Omega_{r}}\bigg(C\abs{\u_{\ep}}^2+K\abs{\w_{\ep}}^2+2R(\u_{\ep}\cdot\w_{\ep})\bigg)\ud{x}\ud{y}\no\\
&\leq&C_0\sum_{i=1}^N\frac{C(b_u^i)^2+K(b_w^i)^2+2R(b_u^i)(b_w^i)}{4\pi}\ln\left(\frac{r}{\ep}\right).\no
\end{eqnarray}
Therefore, we need to consider the scaling $\dfrac{1}{\abs{\ln(\ep)}^{1/2}}$.

Define the functional $J_{\ep}^{(0)}: L^2(\Omega)\times L^2(\Omega)\rt[0,\infty]$ by
\begin{eqnarray*}
J_{\ep}^{(0)}\left[\uu_{\ep},\ww_{\ep}\right]:=\left\{
\begin{array}{ll}
\ds\int_{\Omega_{\ep}}\frac{1}{2}\bigg(C\abs{\uu_{\ep}}^2+K\abs{\ww_{\ep}}^2+2R(\uu_{\ep}\cdot\ww_{\ep})\bigg)\ud{x}\ud{y}\\
\qquad\ \ \text{if}\ \ (\uu_{\ep},\ww_{\ep})=\left(\dfrac{\tilde\uu_{\ep}}{\abs{\ln(\ep)}^{1/2}},\dfrac{\tilde\ww_{\ep}}{\abs{\ln(\ep)}^{1/2}}\right)\ \ \text{for some}\ \ (\tilde\uu_{\ep},\tilde\ww_{\ep})\in H_0^{\ep},\\
\infty\ \ \text{otherwise in}\ \ L^2(\Omega)\times L^2(\Omega).
\end{array}
\right.
\end{eqnarray*}

\begin{theorem}(Compactness)\label{compact theorem}
Assume that (\ref{assumption}) holds and $(\uu_{\ep},\ww_{\ep})\in L^2(\Omega)\times L^2(\Omega)$ satisfy
\begin{eqnarray*}
\sup_{\ep>0}J_{\ep}^{(0)}[\uu_{\ep},\ww_{\ep}]\leq C_0.
\end{eqnarray*}
Then there exists $v_u,v_w\in H^1(\Omega)$ such that up to the extraction of subsequence (non-relabelled),
\begin{eqnarray*}
\left({\bf{1}}_{\Omega_{\ep}}\uu_{\ep},{\bf{1}}_{\Omega_{\ep}}\ww_{\ep}\right)\rightharpoonup(\na v_u,\na v_w)\ \text{in}\ \text{weak}-L^2\ \text{as}\ \ep\rt0.
\end{eqnarray*}
\end{theorem}
\begin{proof}
We use the notation as in the definition of $J_{\ep}^{(0)}$. Using the solution of a single dislocation $(\u_i,\w_i)$ in (\ref{mt temp 12}) and (\ref{mt temp 13}), recalling the definition of $H_0^{\ep}$, we have
\begin{eqnarray*}
\na\times\bigg(\tilde\uu_{\ep}-\sum_{i=1}^N\u_i\bigg)=\na\times\bigg(\tilde\ww_{\ep}-\sum_{i=1}^N\w_i\bigg)=0,\\
\int_{\p B_{\ep}(\v d_i)}\bigg(\tilde\uu_{\ep}-\sum_{i=1}^N\u_i\bigg)\cdot\t\ud{s}=\int_{\p B_{\ep}(\v d_i)}\bigg(\tilde\ww_{\ep}-\sum_{i=1}^N\w_i\bigg)\cdot\t\ud{s}=0.
\end{eqnarray*}
Therefore, using the analysis of Lemma \ref{lemma 1}, we obtain
\begin{eqnarray*}
\tilde\uu_{\ep}-\sum_{i=1}^N\u_i&=&\na \uk_{\ep},\\
\tilde\ww_{\ep}-\sum_{i=1}^N\w_i&=&\na \wk_{\ep},
\end{eqnarray*}
for some $\uk_{\ep},\wk_{\ep}\in H^1(\Omega_{\ep})$. Also, because of (\ref{et temp 11}) and
\begin{eqnarray*}
\int_{\Omega_{\ep}}\bigg(C\abs{\u_{i}}^2+K\abs{\w_{i}}^2+2R(\u_{i}\cdot\w_{i})\bigg)\ud{x}\ud{y}\leq C_0\abs{\ln(\ep)},
\end{eqnarray*}
we know that
\begin{eqnarray*}
\int_{\Omega_{\ep}}\bigg(C\abs{\na \uk_{\ep}}^2+K\abs{\na \wk_{\ep}}^2+2R(\na \uk_{\ep}\cdot\na \wk_{\ep})\bigg)\ud{x}\ud{y}\leq C_0\abs{\ln(\ep)}.
\end{eqnarray*}
In turn, by Poincar\'{e}'s inequality, we have
\begin{eqnarray*}
\nm{\uk_{\ep}}_{H^1(\Omega_{\ep})}+\nm{\wk_{\ep}}_{H^1(\Omega_{\ep})}\leq C\abs{\ln(\ep)}.
\end{eqnarray*}
We can define a natural extension (see \cite{Cermelli.Leoni2005}) of $(\uk_{\ep},\wk_{\ep})$ from $\Omega_{\ep}$ to $\Omega$ as $(\hat \uk_{\ep},\hat \wk_{\ep})$ such that
\begin{eqnarray*}
\nm{\hat \uk_{\ep}}_{H^1(\Omega)}+\nm{\hat \wk_{\ep}}_{H^1(\Omega)}\leq C\abs{\ln(\ep)}.
\end{eqnarray*}
It is easy to see that up to extracting a subsequence,
\begin{eqnarray*}
\left(\frac{\hat \uk_{\ep}}{\abs{\ln(\ep)}^{1/2}},\frac{\hat \wk_{\ep}}{\abs{\ln(\ep)}^{1/2}}\right)\rightharpoonup (v_u,v_w),
\end{eqnarray*}
in weak-$H^1(\Omega)$ for some $(v_u,v_w)\in H^1(\Omega)$.
On the other hand, note that $\u_i,\w_i\notin L^2(\Omega)$, but $\u_i,\w_i\in L^p(\Omega)$ for any $1\leq p<2$, and also
\begin{eqnarray*}
\int_{\Omega_{\ep}}\bigg(\abs{\u_i}^2+\abs{\w_i}^2\bigg)\ud{x}\ud{y}\leq C\abs{\ln(\ep)}.
\end{eqnarray*}
Hence, we know that up to extracting a subsequence
\begin{eqnarray*}
\left(\frac{{\bf{1}}_{\Omega_{\ep}}\u_{i}}{\abs{\ln(\ep)}^{1/2}},\frac{{\bf{1}}_{\Omega_{\ep}}\w_{i}}{\abs{\ln(\ep)}^{1/2}}\right)\rightharpoonup (U^i,W^i),
\end{eqnarray*}
in weak-$L^2(\Omega)$, for some $U^i,W^i\in L^2(\Omega)$. Taking $\phi,\psi\in C_0^{\infty}(\Omega)$, we have
\begin{eqnarray*}
\int_{\Omega_{\ep}}\frac{\u_i\phi+\w_i\psi}{\abs{\ln(\ep)}^{1/2}}\ud{x}\ud{y}
&\leq&\frac{\nm{\u_i}_{L^1}\nm{\phi}_{L^{\infty}}+\nm{\w_i}_{L^1}\nm{\psi}_{L^{\infty}}}{\abs{\ln(\ep)}^{1/2}}\leq\frac{C}{\abs{\ln(\ep)}^{1/2}}\rt0\ \ \text{as}\ \ \ep\rt0.
\end{eqnarray*}
Therefore, we must have $U^i=W^i=0$, i.e.
\begin{eqnarray*}
\left(\frac{{\bf{1}}_{\Omega_{\ep}}\u_{i}}{\abs{\ln(\ep)}^{1/2}},\frac{{\bf{1}}_{\Omega_{\ep}}\w_{i}}{\abs{\ln(\ep)}^{1/2}}\right)\rightharpoonup (0,0).
\end{eqnarray*}
Thus define
\begin{eqnarray*}
\hat\uu_{\ep}:&=&\sum_{i=1}^N\u_i+\na \hat \uk_{\ep},\\
\hat\ww_{\ep}:&=&\sum_{i=1}^N\w_i+\na \hat \wk_{\ep}.
\end{eqnarray*}
such that $\hat\uu_{\ep}=\tilde\uu_{\ep}$ and $\hat\ww_{\ep}=\tilde\ww_{\ep}$ in $\Omega_{\ep}$. In summary, we have shown that
\begin{eqnarray*}
\left(\frac{{\bf{1}}_{\Omega_{\ep}}\tilde\uu_{\ep}}{\abs{\ln(\ep)}^{1/2}},\frac{{\bf{1}}_{\Omega_{\ep}}\tilde\ww_{\ep}}{\abs{\ln(\ep)}^{1/2}}\right)=
\left(\frac{{\bf{1}}_{\Omega_{\ep}}\hat\uu_{\ep}}{\abs{\ln(\ep)}^{1/2}},\frac{{\bf{1}}_{\Omega_{\ep}}\hat\ww_{\ep}}{\abs{\ln(\ep)}^{1/2}}\right)\rightharpoonup
\left(\frac{{\bf{1}}_{\Omega_{\ep}}\na\hat \uk_{\ep}}{\abs{\ln(\ep)}^{1/2}},\frac{{\bf{1}}_{\Omega_{\ep}}\na\hat \wk_{\ep}}{\abs{\ln(\ep)}^{1/2}}\right)\rightharpoonup(\na v_u,\na v_w),
\end{eqnarray*}
in weak-$L^2(\Omega)$.

\end{proof}

\subsection{Zeroth-Order $\Gamma$-Convergence}

\begin{theorem}( $0^{th}$-Order $\Gamma$-Convergence )
Assume that (\ref{assumption}) holds. Define the functional $J_0^{(0)}: L^2(\Omega)\times L^2(\Omega)\rt[0,\infty]$ as
\begin{eqnarray*}
J_{0}^{(0)}[\uu,\ww]:=\left\{
\begin{array}{ll}
\ds\int_{\Omega}\frac{1}{2}\bigg(C\abs{\na v_u}^2+K\abs{\na v_w}^2+2R(\na v_u\cdot\na v_w)\bigg)+\sum_{i=1}^N\frac{C(b_u^i)^2+K(b_w^i)^2+2R(b_u^i)(b_w^i)}{4\pi}&\\
\qquad\qquad\qquad\qquad\qquad\qquad\ \ \text{if}\ \ (\uu,\ww)=\left(\na v_u, \na v_w\right)\ \ \text{for}\ \ \text{some}\ \ v_u,v_w\in H^1(\Omega),\\
\infty\qquad\ \ \text{otherwise in}\ \ L^2(\Omega)\times L^2(\Omega).
\end{array}
\right.
\end{eqnarray*}
Then
\begin{enumerate}
\item
For any sequence of pairs $(\uu_{\ep},\ww_{\ep})\in L^2(\Omega)\times L^2(\Omega)$ such that $\left(\uu_{\ep},\ww_{\ep}\right)\rightharpoonup(\uu,\ww)$ in weak-$L^2(\Omega)$, we have $\liminf_{\ep\rt0}J_{\ep}^{(0)}[\uu_{\ep},\ww_{\ep}]\geq J_{0}^{(0)}[\na v_u,\na v_w]$.
\item
There exists a sequence of pairs $(\uu_{\ep},\ww_{\ep})\in L^2(\Omega)\times L^2(\Omega)$ such that $\left(\uu_{\ep},\ww_{\ep}\right)\rightharpoonup(\uu,\ww)$ in weak-$L^2(\Omega)$, we have $\limsup_{\ep\rt0}J_{\ep}^{(0)}[\uu_{\ep},\ww_{\ep}]\leq J_{0}^{(0)}[\na v_u,\na v_w]$,
\end{enumerate}
which means
\begin{eqnarray*}
J_{\ep}^{(0)}[\uu_{\ep},\ww_{\ep}]\rt J_0^{(0)}[\uu,\ww],
\end{eqnarray*}
in the sense of $\Gamma$-convergence in weak-$L^2(\Omega)$
\end{theorem}
\begin{proof}
We divide the proof into two steps:\\
\ \\
Step 1: $\liminf$.\\
Assume that $(\tilde\uu_{\ep},\tilde\ww_{\ep})\in H_0^{\ep}$, $\left(\dfrac{\tilde\uu_{\ep}}{\abs{\ln(\ep)}^{1/2}},\dfrac{\tilde\ww_{\ep}}{\abs{\ln(\ep)}^{1/2}}\right)\rightharpoonup(\uu,\ww)$ and $J_{0}^{(0)}[\na v_u,\na v_w]$ is finite. Then due to weak convergence in $L^2$ and quadratic $\f$, we know $J_{\ep}^{(0)}[\tilde\uu_{\ep},\tilde\ww_{\ep}]\leq C_0\abs{\ln(\ep)}$. Based on compactness and Theorem \ref{compact theorem}, we must have
\begin{eqnarray*}
\left(\frac{{\bf{1}}_{\Omega_{\ep}}\tilde\uu_{\ep}}{\abs{\ln(\ep)}^{1/2}},\frac{{\bf{1}}_{\Omega_{\ep}}\tilde\ww_{\ep}}{\abs{\ln(\ep)}^{1/2}}\right)\rightharpoonup(\na v_u,\na v_w),
\end{eqnarray*}
for some $v_u,v_w\in H^1(\Omega)$, i.e., we must have
\begin{eqnarray*}
(\uu,\ww)=\left(\na v_u, \na v_w\right).
\end{eqnarray*}
Based on
\begin{eqnarray*}
\tilde\uu_{\ep}=\sum_{i=1}^N\u_i+\na \uk_{\ep},\qquad
\tilde\ww_{\ep}=\sum_{i=1}^N\w_i+\na \wk_{\ep},
\end{eqnarray*}
and the fact that
\begin{eqnarray*}
\left(\dfrac{\u_i}{\abs{\ln(\ep)}^{1/2}},\dfrac{\w_i}{\abs{\ln(\ep)}^{1/2}}\right)\rightharpoonup(0,0)\ \text{in}\ \text{weak}-L^2(\Omega),
\end{eqnarray*}
we deduce that
\begin{eqnarray*}
\left(\dfrac{\na\uk_{\ep}}{\abs{\ln(\ep)}^{1/2}},\dfrac{\na\wk_{\ep}}{\abs{\ln(\ep)}^{1/2}}\right)\rightharpoonup\left(\na v_u, \na v_w\right)\ \text{in}\ \text{weak}-L^2(\Omega).
\end{eqnarray*}
Hence, we obtain
\begin{eqnarray*}
\left(\dfrac{\uk_{\ep}}{\abs{\ln(\ep)}^{1/2}},\dfrac{\wk_{\ep}}{\abs{\ln(\ep)}^{1/2}}\right)\rightharpoonup\left( v_u,  v_w\right)\ \text{in}\ \text{weak}-H^1(\Omega).
\end{eqnarray*}
For $r>\ep$, we write
\begin{eqnarray*}
&&\frac{1}{\abs{\ln(\ep)}}\int_{\Omega_{\ep}}\frac{1}{2}\bigg(C\abs{\tilde\uu_{\ep}}^2+K\abs{\tilde\ww_{\ep}}^2+2R(\tilde\uu_{\ep}\cdot\tilde\ww_{\ep})\bigg)\ud{x}\ud{y}\\
&=&\frac{1}{\abs{\ln(\ep)}}\int_{\Omega_{r}}\frac{1}{2}\bigg(C\abs{\tilde\uu_{\ep}}^2+K\abs{\tilde\ww_{\ep}}^2+2R(\tilde\uu_{\ep}\cdot\tilde\ww_{\ep})\bigg)\ud{x}\ud{y}\no\\
&&+\frac{1}{\abs{\ln(\ep)}}\sum_{i=1}^N\int_{B_r(\v d_{i})\backslash B_{\ep}(\v d_{i})}\frac{1}{2}\bigg(C\abs{\u_i}^2+K\abs{\w_i}^2+2R(\u_i\cdot\w_i)\bigg)\ud{x}\ud{y}\no\\
&&+\frac{1}{\abs{\ln(\ep)}}\sum_{i=1}^N\int_{B_r(\v d_{i})\backslash B_{\ep}(\v d_{i})}\frac{1}{2}\bigg(C\abs{\na \uk_{\ep}}^2+K\abs{\na \wk_{\ep}}^2+2R(\na \uk_{\ep}\cdot\na \wk_{\ep})\bigg)\ud{x}\ud{y}\no\\
&&+\frac{1}{\abs{\ln(\ep)}}\sum_{i\neq j}\int_{B_r(\v d_{i})\backslash B_{\ep}(\v d_{i})}\frac{1}{2}\bigg(C(\u_i\cdot\u_j)+K(\w_i\cdot\w_j)+R(\u_i\cdot\w_j)+R(\w_i\cdot\u_j)\bigg)\ud{x}\ud{y}\no\\
&&+\frac{1}{\abs{\ln(\ep)}}\sum_{i\neq j}\int_{B_r(\v d_{i})\backslash B_{\ep}(\v d_{i})}\frac{1}{2}\bigg(C(\na \uk_{\ep}\cdot\u_j)+K(\na \wk_{\ep}\cdot\w_j)+R(\na \uk_{\ep}\cdot w_j)+R(\na \wk_{\ep}\cdot u_j)\bigg)\ud{x}\ud{y}\no\\
&&+\frac{1}{\abs{\ln(\ep)}}\sum_{i=1}^N\int_{B_r(\v d_{i})\backslash B_{\ep}(\v d_{i})}\frac{1}{2}\bigg(C(\na \uk_{\ep}\cdot\u_i)+K(\na \wk_{\ep}\cdot\w_i)+R(\na \uk_{\ep}\cdot w_i)+R(\na \wk_{\ep}\cdot u_i)\bigg)\ud{x}\ud{y}\no\\
&:=&I+II+III+IV+V+VI.\no
\end{eqnarray*}
By weak lower semi-continuity, we always have
\begin{eqnarray*}
\liminf_{\ep\rt0}I&\geq& \int_{\Omega_{r}}\frac{1}{2}\bigg(C\abs{\na v_u}^2+K\abs{\na v_w}^2+2R(\na v_u\cdot\na v_w)\bigg)\ud{x}\ud{y}\\
&\rt& \int_{\Omega}\frac{1}{2}\bigg(C\abs{\na v_u}^2+K\abs{\na v_w}^2+2R(\na v_u\cdot\na v_w)\bigg)\ud{x}\ud{y},\no
\end{eqnarray*}
as $r\rt0$. On the other hand, a direct computation based on explicit formula (\ref{mt temp 12}) and (\ref{mt temp 13}) reveals
\begin{eqnarray*}
\lim_{\ep\rt0}II=\sum_{i=1}^N\frac{C(b_u^i)^2+K(b_w^i)^2+2R(b_u^i)(b_w^i)}{4\pi},
\end{eqnarray*}
It is easy to see $III\geq0$, which means
\begin{eqnarray*}
\liminf_{\ep\rt0}III\geq0.
\end{eqnarray*}
Since $i\neq j$ in $IV$, then in the integral, at most one of $\u_i$ or $\u_j$ can contribute $\abs{\ln(\ep)}^{1/2}$. A similar argument holds for $\w_i$ and $\w_j$. Hence, we have
\begin{eqnarray*}
\liminf_{\ep\rt0}IV=0,
\end{eqnarray*}
and
\begin{eqnarray*}
\liminf_{\ep\rt0}V=0.
\end{eqnarray*}
Since
\begin{eqnarray*}
\left\{
\begin{array}{rcl}
\na\cdot\u_{i}&=&\na\cdot\w_{i}=0\ \ \text{in}\ \ B_r(\v d_{i})\backslash B_{\ep}(\v d_{i}),\\
\u_{i}\cdot n&=&\w_{i}\cdot n=0\ \ \text{on}\ \ \p B_{\ep}(\v d_{i}),
\end{array}
\right.
\end{eqnarray*}
we may integrate by parts to get
\begin{eqnarray*}
\liminf_{\ep\rt0}VI&=&0.
\end{eqnarray*}
We have shown that
\begin{eqnarray*}
\liminf_{\ep\rt0}J_{\ep}^{(0)}[\uu_{\ep},\ww_{\ep}]\geq J_{0}^{(0)}[\na v_u,\na v_w].
\end{eqnarray*}
Similarly, the compactness and Theorem \ref{compact theorem} imply that when $J_{0}^{(0)}[\na v_u,\na v_w]=\infty$, we must have $J_{\ep}^{(0)}[\uu_{\ep},\ww_{\ep}]\rt\infty$. \\
\ \\
Step 2: $\limsup$.\\
The $J_{0}^{(0)}[\na v_u,\na v_w]=\infty$ case is trivial, we only consider the case when $J_{0}^{(0)}[\na v_u,\na v_w]$ is finite. Define
\begin{eqnarray*}
(\tilde\uu_{\ep},\tilde\ww_{\ep}):=\bigg(\abs{\ln(\ep)}^{1/2}\na v_u+\sum_{i=1}^N\u_i,\ \ \abs{\ln(\ep)}^{1/2}\na v_w+\sum_{i=1}^N\w_i\bigg).
\end{eqnarray*}
We have
\begin{eqnarray*}
\left[\frac{{\bf{1}}_{\Omega_{\ep}}\tilde\uu_{\ep}}{\abs{\ln(\ep)}^{1/2}},\frac{{\bf{1}}_{\Omega_{\ep}}\tilde\ww_{\ep}}{\abs{\ln(\ep)}^{1/2}}\right]\rightharpoonup[\na v_u,\na v_w]\ \text{in}\ \text{weak}-L^2(\Omega),
\end{eqnarray*}
and
\begin{eqnarray*}
&&\frac{1}{\abs{\ln(\ep)}}\int_{\Omega_{\ep}}\frac{1}{2}\bigg(C\abs{\tilde\uu_{\ep}}^2+K\abs{\tilde\ww_{\ep}}^2+2R(\tilde\uu_{\ep}\cdot\tilde\ww_{\ep})\bigg)\ud{x}\ud{y}\\
&=&\int_{\Omega_{\ep}}\frac{1}{2}\bigg(C\abs{\na v_u}^2+K\abs{\na v_w}^2+2R(\na v_u\cdot\na v_w)\bigg)\ud{x}\ud{y}\no\\
&&+\frac{1}{\abs{\ln(\ep)}}\sum_{i=1}^N\int_{\Omega_{\ep}}\frac{1}{2}\bigg(C\abs{\u_i}^2+K\abs{\w_i}^2+2R(\u_i\cdot\w_i)\bigg)\ud{x}\ud{y}\no\\
&&+\frac{1}{\abs{\ln(\ep)}}\sum_{i\neq j}\int_{\Omega_{\ep}}\frac{1}{2}\bigg(C(\u_i\cdot\u_j)+K(\w_i\cdot\w_j)+R(\u_i\cdot\w_j)+R(\w_i\cdot\u_j)\bigg)\ud{x}\ud{y}\no\\
&&+\frac{1}{\abs{\ln(\ep)}^{1/2}}\sum_{i=1}^N\int_{\Omega_{\ep}}\frac{1}{2}\bigg(C(\na v_u\cdot\u_i)+K(\na v_w\cdot\w_i)+R(\na v_u\cdot w_i)+R(\na v_w\cdot u_i)\bigg)\ud{x}\ud{y}\no\\
&:=&I+II+III+IV.\no
\end{eqnarray*}
Estimating it term by term, and using the techniques similar to those in Step 1, we have
\begin{eqnarray*}
\limsup_{\ep\rt0}I&\leq &\int_{\Omega}\frac{1}{2}\bigg(C\abs{\na v_u}^2+K\abs{\na v_w}^2+2R(\na v_u\cdot\na v_w)\bigg)\ud{x}\ud{y},\\
\limsup_{\ep\rt0}II&\leq &\sum_{i=1}^N\frac{C(b_u^i)^2+K(b_w^i)^2+2R(b_u^i)(b_w^i)}{4\pi},\no\\
\lim_{\ep\rt0}III&= &0,\\
\lim_{\ep\rt0}IV&= &0,
\end{eqnarray*}
and conclude that
\begin{eqnarray*}
\limsup_{\ep\rt0}J_{\ep}^{(0)}[\uu_{\ep},\ww_{\ep}]\leq J_{0}^{(0)}[\uu,\ww].
\end{eqnarray*}
\end{proof}

By Theorem \ref{prelim} and the basis properties of $\Gamma$-convergence, we can naturally obtain an approximation of energy.
\begin{corollary}\label{corollary 1}
Assume that (\ref{assumption}) holds. We have
\begin{eqnarray*}
\inf_{ \uu, \ww} J_{0}^{(0)}[\uu,\ww]=\sum_{i=1}^N\frac{C(b_u^i)^2+K(b_w^i)^2+2R(b_u^i)(b_w^i)}{4\pi}.
\end{eqnarray*}
Assume $(\u_{\ep}',\w_{\ep}')$ is the minimizer of $J_{\ep}^{(0)}$, then we have
\begin{eqnarray*}
J_{\ep}^{(0)}[\u_{\ep}',\w_{\ep}']=E_0+o(1),
\end{eqnarray*}
where the rescaled leading-order energy
\begin{eqnarray}\label{et 001}
E_0=\sum_{i=1}^N\frac{C(b_u^i)^2+K(b_w^i)^2+2R(b_u^i)(b_w^i)}{4\pi}.
\end{eqnarray}
\end{corollary}

\subsection{First-Order $\Gamma$-Convergence}

Since the leading order energy $E_0$ only concerns with magnitude of the Burger's vectors and loses information about the dislocation position, we need more detailed analysis of convergence and selection process, which leads us to considering the first-order $\Gamma$-convergence.

Now we get rid of the rescaling $\dfrac{1}{\abs{\ln(\ep)}^{1/2}}$. Define the functional $J_{\ep}^{(1)}: L^2(\Omega)\times L^2(\Omega)\rt[0,\infty]$ as
\begin{eqnarray*}
J_{\ep}^{(1)}\left[\tilde\uu_{\ep},\tilde\ww_{\ep}\right]:=\left\{
\begin{array}{ll}
\ds\int_{\Omega_{\ep}}\frac{1}{2}\bigg(C\abs{\tilde\uu_{\ep}}^2+K\abs{\tilde\ww_{\ep}}^2+2R(\tilde\uu_{\ep}\cdot\tilde\ww_{\ep})\bigg)\ud{x}\ud{y}-\abs{\ln(\ep)}\inf_{ \uu, \ww} J_0^{(0)}[ \uu, \ww]\\
\qquad\qquad\qquad\qquad\qquad\qquad\qquad\qquad\qquad\qquad\qquad\ \ \text{if}\ \ (\tilde\uu_{\ep},\tilde\ww_{\ep})\in H_0^{\ep},\\
\infty\ \ \text{otherwise in}\ \ L^2(\Omega)\times L^2(\Omega).
\end{array}
\right.
\end{eqnarray*}

\begin{theorem}( $1^{st}$-Order $\Gamma$-Convergence )
Assume that (\ref{assumption}) holds. Define the functional $J_0^{(1)}: L^2(\Omega)\times L^2(\Omega)\rt[0,\infty]$ as
\begin{eqnarray*}
J_0^{(1)}[ \tilde\uu, \tilde\ww]:=\left\{
\begin{array}{ll}
E_{\textnormal{self}}+E_{\textnormal{int}}+E_{\textnormal{elastic}}\ \ \text{if}\ \ ( \tilde\uu, \tilde\ww)=\ds\left(\na v_u+\sum_{i=1}^N\u_i, \na v_w+\sum_{i=1}^N\w_i\right),\\
\qquad\qquad\qquad\qquad\qquad\qquad\qquad\qquad\text{for}\ \ \text{some}\ \ (v_u,v_w)\in H^1(\Omega),\\
\infty\qquad\ \ \text{otherwise in}\ \ L^2(\Omega)\times L^2(\Omega),
\end{array}
\right.
\end{eqnarray*}
where
\begin{eqnarray*}
E_{\textnormal{self}}&:=&\sum_{i=1}^N\int_{\Omega\backslash B_r(\v d_i)}\frac{1}{2}\bigg(C\abs{\u_{i}}^2+K\abs{\w_{i}}^2+2R(\u_{i}\cdot\w_{i})\bigg)\ud{x}\ud{y}\\
&&+\sum_{i=1}^N\frac{(C(b_u^i)^2+K(b_w^i)^2+2R(b_u^i)(b_w^i))}{4\pi}\ln(r),\no\\
E_{\textnormal{int}}&:=&\sum_{i=1}^{N-1}\sum_{j=i}^N\int_{\Omega}\bigg(C(\u_i\cdot\u_j)+K(\w_i\cdot\w_j)+R(\u_i\cdot\w_j)+R(\u_j\cdot\w_i)\bigg),\\
E_{\textnormal{elastic}}&:=&J[\na v_u,\na v_w]+\sum_{i=1}^N\int_{\p\Omega}\bigg(v_u(C\u_i+R\w_i)+v_w(K\w_i+R\u_i)\bigg)\cdot n\ud{s}.
\end{eqnarray*}
Then
\begin{enumerate}
\item
For any sequence of pairs $(\tilde\uu_{\ep},\tilde\ww_{\ep})\in L^2(\Omega)\times L^2(\Omega)$ such that $\left(\tilde\uu_{\ep},\tilde\ww_{\ep}\right)\rightharpoonup(\tilde\uu,\tilde\ww)$ in weak-$L^2(\Omega)$, we have $\liminf_{\ep\rt0}J_{\ep}^{(1)}[\tilde\uu_{\ep},\tilde\ww_{\ep}]\geq J_{0}^{(1)}[\na v_u,\na v_w]$.
\item
There exists a sequence of pairs $(\tilde\uu_{\ep},\tilde\ww_{\ep})\in L^2(\Omega)\times L^2(\Omega)$ such that $\left(\tilde\uu_{\ep},\tilde\ww_{\ep}\right)\rightharpoonup(\tilde\uu,\tilde\ww)$ in weak-$L^2(\Omega)$, we have $\limsup_{\ep\rt0}J_{\ep}^{(1)}[\tilde\uu_{\ep},\tilde\ww_{\ep}]\leq J_{0}^{(1)}[\na v_u,\na v_w]$,
\end{enumerate}
which means
\begin{eqnarray*}
J_{\ep}^{(1)}[\tilde\uu_{\ep},\tilde\ww_{\ep}]\rt J_0^{(1)}[\tilde\uu,\tilde\ww],
\end{eqnarray*}
in the sense of $\Gamma$-convergence in weak-$L^2(\Omega)$
\end{theorem}
\begin{proof}
We naturally have
\begin{eqnarray*}
J_{\ep}^{(1)}\left[ \tilde\uu_{\ep}, \tilde\ww_{\ep}\right]&=&\left\{
\begin{array}{ll}
\ds\int_{\Omega_{\ep}}\frac{1}{2}\bigg(C\abs{ \tilde\uu_{\ep}}^2+K\abs{ \tilde\ww_{\ep}}^2+2R( \tilde\uu_{\ep}\cdot \tilde\ww_{\ep})\bigg)\ud{x}\ud{y}-\ds\abs{\ln(\ep)}\sum_{i=1}^N\dfrac{C(b_u^i)^2+K(b_w^i)^2+2R(b_u^i)(b_w^i)}{4\pi}\\
\qquad\qquad\qquad\qquad\qquad\qquad\qquad\qquad\qquad\qquad\qquad\qquad\qquad\qquad
\ \ \text{if}\ \ ( \tilde\uu_{\ep}, \tilde\ww_{\ep})\in H_0^{\ep},\\
\infty\qquad\ \ \text{otherwise in}\ \ L^2(\Omega)\times L^2(\Omega).
\end{array}
\right.
\end{eqnarray*}
We first prove the $\liminf$ part. Consider weakly convergent sequence
\begin{eqnarray*}
( \tilde\uu_{\ep}, \tilde\ww_{\ep})={\bf{1}}_{\Omega_{\ep}}\bigg(\na \uk_{\ep}+\sum_{i=1}^N\u_i,\ \ \na \wk_{\ep}+\sum_{i=1}^N\w_i\bigg)\rightharpoonup (\tilde\uu,\tilde\ww)=\ds\left(\na v_u+\sum_{i=1}^N\u_i, \na v_w+\sum_{i=1}^N\w_i\right),
\end{eqnarray*}
Direct computation using (\ref{mt temp 12}) and (\ref{mt temp 13}) yields
\begin{eqnarray*}
{\bf{1}}_{\Omega_{\ep}}\bigg(\sum_{i=1}^N\u_i,\ \ \sum_{i=1}^N\w_i\bigg)\rightharpoonup\bigg(\sum_{i=1}^N\u_i,\ \ \sum_{i=1}^N\w_i\bigg).
\end{eqnarray*}
Naturally, we have
\begin{eqnarray*}
(\na \uk_{\ep},\na \wk_{\ep})\rightharpoonup(\na v_u,\na v_w).
\end{eqnarray*}
Hence, weak convergence yields boundedness $\nm{\na\uk_{\ep}}_{L^2(\Omega)}+\nm{\na\wk_{\ep}}_{L^2(\Omega)}\leq C'$ for some constant $C'$ independent of $\ep$.
We may decompose
\begin{eqnarray*}
&&\int_{\Omega_{\ep}}\frac{1}{2}\bigg(C\abs{ \tilde\uu_{\ep}}^2+K\abs{ \tilde\ww_{\ep}}^2+2R( \tilde\uu_{\ep}\cdot \tilde\ww_{\ep})\bigg)\ud{x}\ud{y}
-\abs{\ln(\ep)}\sum_{i=1}^N\dfrac{C(b_u^i)^2+K(b_w^i)^2+2R(b_u^i)(b_w^i)}{4\pi}\\
&=&\bigg(\sum_{i=1}^N\int_{\Omega_{\ep}}\frac{1}{2}\bigg(C\abs{\u_i}^2+K\abs{\w_i}^2+2R(\u_i\cdot\w_i)\bigg)\ud{x}\ud{y}
-\abs{\ln(\ep)}\sum_{i=1}^N\dfrac{C(b_u^i)^2+K(b_w^i)^2+2R(b_u^i)(b_w^i)}{4\pi}\bigg)\no\\
&&+\sum_{i\neq j}\int_{\Omega_{\ep}}\frac{1}{2}\bigg(C(\u_i\cdot\u_j)+K(\w_i\cdot\w_j)+R(\u_i\cdot\w_j)+R(\w_i\cdot\u_j)\bigg)\ud{x}\ud{y}\no\\
&&+\int_{\Omega_{\ep}}\frac{1}{2}\bigg(C\abs{\na \uk_{\ep}}^2+K\abs{\na \wk_{\ep}}^2+2R(\na \uk_{\ep}\cdot\na \wk_{\ep})\bigg)\ud{x}\ud{y}\no\\
&&+\sum_{i=1}^N\int_{\Omega_{\ep}}\frac{1}{2}\bigg(C(\na \uk_{\ep}\cdot\u_i)+K(\na \wk_{\ep}\cdot\w_i)+R(\na \uk_{\ep}\cdot w_i)+R(\na \wk_{\ep}\cdot u_i)\bigg)\ud{x}\ud{y}\no\\
&:=&I+II+III+IV.\no
\end{eqnarray*}
Here the argument is similar to that in the proof of $0^{th}$-order $\Gamma$-convergence, so we only describe the main strategy. For $I$, decompose $\Omega_{\ep}=\Omega_{r}\cup \left(\Omega_{\ep}\backslash\Omega_{r}\right)$ for some $r>\ep$, i.e.
\begin{eqnarray*}
I&=&\sum_{i=1}^N\int_{\Omega_{r}}\frac{1}{2}\bigg(C\abs{\u_i}^2+K\abs{\w_i}^2+2R(\u_i\cdot\w_i)\bigg)\ud{x}\ud{y}
+\sum_{i=1}^N\int_{\Omega_{\ep}\backslash\Omega_{r}}\frac{1}{2}\bigg(C\abs{\u_i}^2+K\abs{\w_i}^2+2R(\u_i\cdot\w_i)\bigg)\ud{x}\ud{y}\\
&&-\abs{\ln(\ep)}\sum_{i=1}^N\dfrac{C(b_u^i)^2+K(b_w^i)^2+2R(b_u^i)(b_w^i)}{4\pi}\no
\end{eqnarray*}
Direct computation using (\ref{mt temp 12}) and (\ref{mt temp 13}) reveals that
\begin{eqnarray*}
&&\lim_{\ep\rt0}\bigg(\sum_{i=1}^N\int_{\Omega_{\ep}\backslash\Omega_{r}}\frac{1}{2}\bigg(C\abs{\u_i}^2+K\abs{\w_i}^2+2R(\u_i\cdot\w_i)\bigg)\ud{x}\ud{y}\bigg)\\
&=&\abs{\ln(\ep)}\sum_{i=1}^N\dfrac{C(b_u^i)^2+K(b_w^i)^2+2R(b_u^i)(b_w^i)}{4\pi}+\sum_{i=1}^N\frac{(C(b_u^i)^2+K(b_w^i)^2+2R(b_u^i)(b_w^i))}{4\pi}\ln(r).\no
\end{eqnarray*}
Hence, we know
\begin{eqnarray*}
\lim_{\ep\rt0}I=E_{\textnormal{self}}.
\end{eqnarray*}
Similarly, a direct computation using (\ref{mt temp 12}) and (\ref{mt temp 13}) shows that
\begin{eqnarray*}
\lim_{\ep\rt0}II=E_{\textnormal{int}}.
\end{eqnarray*}
Based on weak convergence $(\na \uk_{\ep},\na \wk_{\ep})\rightharpoonup(\na v_u,\na v_w)$ and weak lower semi-continuity, we know that
\begin{eqnarray*}
\liminf_{\ep\rt0}III\geq J[\na v_u,\na v_w].
\end{eqnarray*}
Finally, after integrating by parts, by weak convergence and the equations (\ref{temp 1}) satisfied by $(\u_i,\w_i)$, we know that
\begin{eqnarray*}
\lim_{\ep\rt0}IV=\displaystyle\sum_{i=1}^N\int_{\p\Omega}\bigg(v_u(C\u_i+R\w_i)+v_w(K\w_i+R\u_i)\bigg)\cdot n\ud{s}.
\end{eqnarray*}
Therefore,
\begin{eqnarray*}
\liminf_{\ep\rt0}(III+IV)\geq E_{\textnormal{elastic}}.
\end{eqnarray*}
To summarize, this concludes the proof of the $\liminf$ part.\\
\ \\
For the $\limsup$ part, consider the sequence
\begin{eqnarray*}
( \tilde\uu_{\ep}, \tilde\ww_{\ep})={\bf{1}}_{\Omega_{\ep}}\bigg(\na v_u+\sum_{i=1}^N\u_i,\ \ \na v_w+\sum_{i=1}^N\w_i\bigg),
\end{eqnarray*}
and we have
\begin{eqnarray*}
( \tilde\uu_{\ep}, \tilde\ww_{\ep})\rightharpoonup\bigg(\na v_u+\sum_{i=1}^N\u_i,\ \ \na v_w+\sum_{i=1}^N\w_i\bigg).
\end{eqnarray*}
Therefore, a direct computation using explicit formula (\ref{mt temp 12}) and (\ref{mt temp 13}) justifies the result, and thus the $\Gamma$-convergence holds.
\end{proof}

Similar to the analysis of Corollary \ref{corollary 1}, Theorem \ref{prelim} and the basic property of $\Gamma$-convergence justify a more detailed energy approximation.
\begin{corollary}\label{corollary 2}
Assume that (\ref{assumption}) holds. We have
\begin{eqnarray*}
\inf_{ \tilde\uu, \tilde\ww} J_{0}^{(1)}[\tilde\uu,\tilde\ww]=F_{\textnormal{self}}+F_{\textnormal{int}}+F_{\textnormal{elastic}},
\end{eqnarray*}
where
\begin{eqnarray}\label{et 002}
F_{\textnormal{self}}:&=&\sum_{i=1}^N\int_{\Omega\backslash B_r(\v d_i)}\frac{1}{2}\bigg(C\abs{\u_{i}}^2+K\abs{\w_{i}}^2+2R(\u_{i}\cdot\w_{i})\bigg)\ud{x}\ud{y}\\
&&+\sum_{i=1}^N\frac{(C(b_u^i)^2+K(b_w^i)^2+2R(b_u^i)(b_w^i))}{4\pi}\ln(r),\no\\
F_{\textnormal{int}}:&=&\sum_{i=1}^{N-1}\sum_{j=i}^N\int_{\Omega}\bigg(C(\u_i\cdot\u_j)+K(\w_i\cdot\w_j)+R(\u_i\cdot\w_j)+R(\u_j\cdot\w_i)\bigg),\no\\
F_{\textnormal{elastic}}:&=&J[\na u_{0},\na w_{0}]+\sum_{i=1}^N\int_{\p\Omega}\bigg(u_{0}(C\u_i+R\w_i)+w_{0}(K\w_i+R\u_i)\bigg)\cdot n\ud{s},\no
\end{eqnarray}
in which $(u_0,w_0)$ is the minimizer of
\begin{eqnarray*}
I[v_u,v_w]=J[\na v_u,\na v_w]+\sum_{i=1}^N\int_{\p\Omega}\bigg(v_u(C\u_i+R\w_i)+v_w(K\w_i+R\u_i)\bigg)\cdot n\ud{s}.
\end{eqnarray*}
Assume $(\tilde\u_{\ep}',\tilde\w_{\ep}')\in H_0^{\ep}$ is the minimizer of $J_{\ep}^{(1)}$, then we have
\begin{eqnarray*}
J_{\ep}^{(1)}[\tilde\u_{\ep}',\tilde\w_{\ep}']=F_{\textnormal{self}}+F_{\textnormal{int}}+F_{\textnormal{elastic}}+o(1).
\end{eqnarray*}
\end{corollary}

\begin{remark}
The existence and uniqueness of minimizer $(u_0,w_0)$ can be proved using a similar argument as in Section 2.3 and 2.4.
\end{remark}

\begin{remark}
We can show that $F_{\textnormal{self}}$ is independent of the choice of $r$. Assume $r'<\bar r$, say $r'< r$, then we have
\begin{eqnarray*}
&&\sum_{i=1}^N\int_{\Omega\backslash B_{r'}(\v d_i)}\frac{1}{2}\bigg(C\abs{\u_{i}}^2+K\abs{\w_{i}}^2+2R(\u_{i}\cdot\w_{i})\bigg)\ud{x}\ud{y}\\
&&+\sum_{i=1}^N(C(b_u^i)^2+K(b_w^i)^2+2R(b_u^i)(b_w^i))\frac{1}{4\pi}\ln(r')\no\\
&=&\sum_{i=1}^N\int_{\Omega\backslash B_r(\v d_i)}\frac{1}{2}\bigg(C\abs{\u_{i}}^2+K\abs{\w_{i}}^2+2R(\u_{i}\cdot\w_{i})\bigg)\ud{x}\ud{y}\no\\
&&+\sum_{i=1}^N\int_{B_r(\v d_i)\backslash B_{r'}(\v d_i)}\frac{1}{2}\bigg(C\abs{\u_{i}}^2+K\abs{\w_{i}}^2+2R(\u_{i}\cdot\w_{i})\bigg)\ud{x}\ud{y}\no\\
&&+\sum_{i=1}^N(C(b_u^i)^2+K(b_w^i)^2+2R(b_u^i)(b_w^i))\frac{1}{4\pi}\ln(r')\no\\
&=&\sum_{i=1}^N\int_{\Omega\backslash B_r(\v d_i)}\frac{1}{2}\bigg(C\abs{\u_{i}}^2+K\abs{\w_{i}}^2+2R(\u_{i}\cdot\w_{i})\bigg)\ud{x}\ud{y}\no\\
&&+\sum_{i=1}^N(C(b_u^i)^2+K(b_w^i)^2+2R(b_u^i)(b_w^i))\frac{1}{4\pi}\ln\left(\frac{r}{r'}\right)+\sum_{i=1}^N(C(b_u^i)^2+K(b_w^i)^2+2R(b_u^i)(b_w^i))\frac{1}{4\pi}\ln(r')\no\\
&=&\sum_{i=1}^N\int_{\Omega\backslash B_r(\v d_i)}\frac{1}{2}\bigg(C\abs{\u_{i}}^2+K\abs{\w_{i}}^2+2R(\u_{i}\cdot\w_{i})\bigg)\ud{x}\ud{y}\no\\
&&+\sum_{i=1}^N(C(b_u^i)^2+K(b_w^i)^2+2R(b_u^i)(b_w^i))\frac{1}{4\pi}\ln(r).\no
\end{eqnarray*}
Hence, choosing $r'$ or $r$ gives exactly the same $F_{\textnormal{self}}$.
\end{remark}

\subsection{Minimizer and Energy Structure}

Combining Corollary \ref{corollary 1} and Corollary \ref{corollary 2}, we can describe the structure of minimizer and energy.

\begin{theorem}
Assume that (\ref{assumption}) holds. The problem (\ref{problem.}) admits a unique solution
\begin{eqnarray*}
\u_{\ep}=\sum_{i=1}^N\u_i+\na u_{\ep},\qquad \w_{\ep}=\sum_{i=1}^N\w_i+\na w_{\ep},
\end{eqnarray*}
where
\begin{eqnarray*}
\u_i&=&\frac{b_u^i}{2\pi}\frac{1}{(x-x_i)^2+(y-y_i)^2}\bigg(-(y-y_i),(x-x_i)\bigg),\\
\w_i&=&\frac{b_w^i}{2\pi}\frac{1}{(x-x_i)^2+(y-y_i)^2}\bigg(-(y-y_i),(x-x_i)\bigg),
\end{eqnarray*}
and $(u_{\ep},w_{\ep})$ is the unique minimizer of
\begin{eqnarray*}
I_{\ep}[u_{\ep},w_{\ep}]:&=&J_{\ep}[\na u_{\ep},\na w_{\ep}]+\sum_{i=1}^N\int_{\p\Omega}\bigg(u_{\ep}(C\u_i+R\w_i)+w_{\ep}(K\w_i+R\u_i)\bigg)\cdot n\ud{s}\no\\
&&-\sum_{i=1}^N\sum_{j\neq i}\int_{\p B_{\ep}(x_i,y_i)}\bigg(u_{\ep}(C\u_j+R\w_j)+w_{\ep}(K\w_j+R\u_j)\bigg)\cdot n\ud{s},\no
\end{eqnarray*}
subject to
$\displaystyle\int_{B}u_{\ep}\ud{x}\ud{y}=0$ and $\displaystyle\int_{B}w_{\ep}\ud{x}\ud{y}=0$
for some ball $B\subset\Omega_{\ep}$, with $n$ the outward unit normal vector on $\p\Omega$.

Furthermore, $(\u_{\ep},\w_{\ep})$ converges in weak-$L^2(\Omega)$ as $\ep\rt0$ to $(\u_0,\w_0)$ where
\begin{eqnarray*}
\u_{0}=\sum_{i=1}^N\u_i+\na u_{0},\qquad \w_{0}=\sum_{i=1}^N\w_i+\na w_{0}.
\end{eqnarray*}
and $[u_0,w_0]$ is the unique minimizer of
\begin{eqnarray*}
I_{0}[u_0,w_0]&=&J[\na u_{0},\na w_{0}]+\sum_{i=1}^N\int_{\p\Omega}\bigg(u_{0}(C\u_i+R\w_i)+w_{0}(K\w_i+R\u_i)\bigg)\cdot n\ud{s},
\end{eqnarray*}
subject to
$\displaystyle\int_{B}u_0\ud{x}\ud{y}=0$ and $\displaystyle\int_{B}w_0\ud{x}\ud{y}=0$
for some ball $B\subset\Omega_{\ep}$
\end{theorem}

\begin{proof}
The existence and uniqueness of minimizer have been shown in Theorem \ref{prelim}. $\Gamma$-convergence naturally yields that minimizer of $J_{\ep}^{(1)}$ goes to minimizer of $J_0^{(1)}$. Hence, this result is obvious.
\end{proof}

\begin{theorem}
Assume that (\ref{assumption}) holds. We have
\begin{eqnarray*}
J_{\ep}[\u_{\ep},\w_{\ep}]=\int_{\Omega_{\ep}}\f[\u_{\ep},\w_{\ep}]\ud{x}\ud{y}&=&E_0\ln\bigg(\frac{1}{\ep}\bigg)+F+o(1),
\end{eqnarray*}
where the core energy $E_0$ is defined in (\ref{et 001})
and the renormalized energy $F=F_{\textnormal{self}}+F_{\textnormal{int}}+F_{\textnormal{elastic}}$ is defined in (\ref{et 002}).
\end{theorem}
\begin{proof}
We directly compute
\begin{eqnarray*}
J_{\ep}[ \u_{\ep}, \w_{\ep}]&=&\abs{\ln(\ep)}J_{\ep}^{(0)}\left[\dfrac{ \u_{\ep}}{\abs{\ln(\ep)}^{1/2}},\dfrac{ \w_{\ep}}{\abs{\ln(\ep)}^{1/2}}\right]
=E_0\ln\bigg(\frac{1}{\ep}\bigg)+J_{\ep}^{(1)}[ \u_{\ep}, \w_{\ep}]=E_0\ln\bigg(\frac{1}{\ep}\bigg)+F+o(1).
\end{eqnarray*}
\end{proof}

\section{Application of Renormalized Energy}

\subsection{Interaction between Dislocations}

In this section, we will prove that the energy related to interaction between dislocation $F_{\textnormal{int}}$ obeys the inverse logarithmical law of the distance between two dislocations.
\begin{theorem}
Assume that (\ref{assumption}) holds. We have
\begin{eqnarray*}
F_{\textnormal{int}}&=&\sum_{i=1}^{N-1}\sum_{j=i}^N\frac{Cb_u^ib_u^j+Kb_w^ib_w^j+Rb_u^ib_w^j+Rb_w^ib_u^j}{2\pi}\ln\left(\frac{1}{\abs{\v d_i-\v d_j}}\right)+O(1).
\end{eqnarray*}
\end{theorem}
\begin{proof}
Since
\begin{eqnarray*}
F_{\textnormal{int}}&=&\sum_{i=1}^{N-1}\sum_{j=i}^N\int_{\Omega}\bigg(C(\u_i\cdot\u_j)+K(\w_i\cdot\w_j)+R(\u_i\cdot\w_j)+R(\u_j\cdot\w_i)\bigg),
\end{eqnarray*}
let $\v d_i,\v d_j\in\Omega$ and let $\gamma$ be a segment of line that connects $\v d_j$ to $\p\Omega$ and is parallel to $\v d_i-\v d_j$. We rewrite
\begin{eqnarray*}
\gamma=\{\v d\in\Omega: \v d=\v d_j+s(\v d_j-\v d_i)\ \ \text{for}\ \ s\in[0,\bar s]\}
\end{eqnarray*}
where $\bar s$ depends on the distance between $\d d_i$, $\v d_j$ and $\p\Omega$. Let
\begin{eqnarray*}
\v m=\left(\frac{\v d_j-\v d_i}{\abs{\v d_j-\v d_i}}\right)^{\bot}
\end{eqnarray*}
indicate the unit vector perpendicular to $\v d_j-\v d_i$. Note that although $\Omega\backslash\{\v d_j\}$ is not simply connected, $\Omega\backslash\gamma$ is. Hence, due to curl-free condition, there exist $U$ and $W$ such that $\u_j=\na U$ and $\w_j=\na W$ in $\Omega\backslash\gamma$ such that $[U]=b_u^j$ and $[W]=b_w^j$, where $[\cdot]$ denotes the jump across $\gamma$. By the divergence theorem, we have
\begin{eqnarray*}
&&\int_{\Omega}\bigg(C(\u_i\cdot\u_j)+K(\w_i\cdot\w_j)+R(\u_i\cdot\w_j)+R(\w_i\cdot\u_j)\bigg)\\
&=&\int_{\Omega\backslash\gamma}\bigg(C(\u_i\cdot\na U)+K(\w_i\cdot\na W)+R(\u_i\cdot\na W)+R(\w_i\cdot\na U)\bigg)\no\\
&=&\int_{\p\Omega}\bigg(C\u_i(U\cdot n)+K\w_i(W\cdot n)+R\u_i(W\cdot n)+R\w_i(U\cdot n)\bigg)\ud{s}\no\\
&&-\int_{\gamma}\bigg(C\u_i[U]+K\w_i[W]+R\u_i[W]+R\w_i[U]\bigg)\cdot\v m\ud{s}.\no
\end{eqnarray*}
The first integral is bounded since all quantities are uniformly bounded on $\p\Omega$. For the second integral, we estimate
\begin{eqnarray*}
&&-\int_{\gamma}\bigg(C\u_i[U]+K\w_i[W]+R\u_i[W]+R\w_i[U]\bigg)\cdot\v m\ud{s}\\
&=&\int_{\gamma}\bigg(C\u_ib_u^j+K\w_ib_w^j+R\u_ib_w^j+R\w_ib_u^j\bigg)\cdot\v m\ud{s}.\no
\end{eqnarray*}
By explicit formula (\ref{mt temp 12}) and (\ref{mt temp 13}), we know
\begin{eqnarray*}
\u_i(\v d)=-\frac{b_u^i}{2\pi}\frac{\v m}{\abs{\v d-\v d_i}},\qquad
\w_i(\v d)=-\frac{b_w^i}{2\pi}\frac{\v m}{\abs{\v d-\v d_i}}.
\end{eqnarray*}
Hence, we have
\begin{eqnarray*}
&&\int_{\gamma}\bigg(C\u_ib_u^j+K\w_ib_w^j+R\u_ib_w^j+R\w_ib_u^j\bigg)\cdot\v m\ud{s}\\
&=&\int_{\gamma}\frac{Cb_u^ib_u^j+Kb_w^ib_w^j+Rb_u^ib_w^j+Rb_w^ib_u^j}{2\pi}\frac{1}{\abs{\v d-\v d_i}}\ud{s}\no\\
&=&\frac{Cb_u^ib_u^j+Kb_w^ib_w^j+Rb_u^ib_w^j+Rb_w^ib_u^j}{2\pi}\int_0^{\bar s}\frac{1}{\abs{\v d_i-\v d_j}+s}\ud{s}\no\\
&=&\frac{Cb_u^ib_u^j+Kb_w^ib_w^j+Rb_u^ib_w^j+Rb_w^ib_u^j}{2\pi}\left(\ln\left(\frac{1}{\abs{\v d_i-\v d_j}}\right)+\ln\bigg(\abs{\v d_i-\v d_j}+\bar s\bigg)\right).\no
\end{eqnarray*}
The result follows since we always have $\bar s>0$.
\end{proof}

\subsection{Peach-K\"{o}hler Force}

The Peach-K\"{o}hler Force acting on the dislocation $\v d_k$ is given by $\na_{\v d_k}F$ (see \cite{Gurtin1995}). In this section, we will show its relation with the renormalized energy. Here we first present three lemmas proved in \cite{Cermelli.Leoni2005}.
\begin{lemma}
Define
\begin{eqnarray*}
D_k^Vf(\v d)=\frac{\ud{}}{\ud{\theta}}f(\v d; \v d_1,\v d_2,\cdots,\v d_k+\theta\v V,\cdots,\v d_N)\bigg|_{\theta=0}.
\end{eqnarray*}
Then we have
\begin{eqnarray*}
D_k^V\u_k&=&0\ \ \text{for}\ \ k\neq i,\\
D_k^V\w_k&=&0\ \ \text{for}\ \ k\neq i,\\
D_k^V\u_k&=&-D\u_k\cdot\v V=-\na(\u_k\cdot\v V),\\
D_k^V\w_k&=&-D\w_k\cdot\v V=-\na(\w_k\cdot\v V),\\
D_k^V\u_0&=&\na U=\na(D_k^Vu_0-\u_k\cdot\v V),\\
D_k^V\w_0&=&\na W=\na(D_k^Vw_0-\w_k\cdot\v V)
\end{eqnarray*}
where $D$ is the derivative with respect to $\v d$.
\end{lemma}
\begin{lemma}
We have
\begin{eqnarray*}
\frac{\ud{}}{\ud{\theta}}\int_{B_{\ep}(\v d_0+\theta\v V)}f(\v d,\theta)\ud{x}\ud{y}\bigg|_{\theta=0}&=&\int_{B_{\ep}(\v d_0)}D_{\theta}f(\v d,0)\ud{x}\ud{y}\\
&=&\int_{B_{\ep}(\v d_0)}\p_{\theta}f(\v d,0)\ud{x}\ud{y}+\int_{\p B_{\ep}(\v d_0)}f(\v d,0)\v V\cdot n\ud{s},\no\\
\frac{\ud{}}{\ud{\theta}}\int_{\p B_{\ep}(\v d_0+\theta\v V)}g(\v d,\theta)\ud{s}\bigg|_{\theta=0}&=&\int_{\p B_{\ep}(\v d_0)}D_{\theta}g(\v d,0)\ud{s},\\
\frac{\ud{}}{\ud{\theta}}\int_{\Omega\backslash B_{\ep}(\v d_0+\theta\v V)}r(\v d,\theta)\ud{x}\ud{y}\bigg|_{\theta=0}&=&\int_{\Omega\backslash B_{\ep}(\v d_0)}\p_{\theta}r(\v d,0)\ud{x}\ud{y}-\int_{\p B_{\ep}(\v d_0)}r(\v d,0)\v V\cdot n\ud{s},\no
\end{eqnarray*}
where $D_{\theta}=\p_{\theta}+\v V\cdot\na$.
\end{lemma}
\begin{lemma}
We have
\begin{eqnarray*}
D_{\theta}\u_i(\v d; \v d_i+\theta\v V)=0,
\end{eqnarray*}
for any $\v V$.
\end{lemma}
Now we can prove the main result.
\begin{theorem}
Assume that (\ref{assumption}) holds. The Peach-K\"{o}hler force acting at $\v d_k$ is given by
\begin{eqnarray*}
\na_{\v d_k}F
&=&-\int_{\p B_r(\v d_k)}\bigg(\f[\u_0,\w_0]\i-(C\u_0\otimes\u_0+K\w_0\otimes\w_0+R\u_0\otimes\w_0+R\w_0\otimes\u_0)\bigg)\cdot n\ud{s},
\end{eqnarray*}
for $r<\dfrac{1}{2}\min_k\Big(dist(\v d_k,\p\Omega)\Big)$.
\end{theorem}
\begin{proof}
We decompose the renormalized energy
\begin{eqnarray*}
F(\v d_1,\v d_2,\cdots,\v d_N)&=&G(\v d_1,\v d_2,\cdots,\v d_N)+H(\v d_1,\v d_2,\cdots,\v d_N),
\end{eqnarray*}
where
\begin{eqnarray*}
G(\v d_1,\v d_2,\cdots,\v d_N)&:=&\int_{\Omega_{\ep}}\frac{1}{2}\bigg(C\abs{\u_{0}}^2+K\abs{\w_{0}}^2+2R(\u_{0}\cdot\w_{0})\bigg)\ud{x}\ud{y},\\
H(\v d_1,\v d_2,\cdots,\v d_N)&:=&\sum_{i=1}^N\sum_{m\neq i}\int_{B_{\ep}(\v d_m)}\frac{1}{2}\bigg(C\abs{\u_{i}}^2+K\abs{\w_{i}}^2+2R(\u_{i}\cdot\w_{i})\bigg)\ud{x}\ud{y}\\
&&+\sum_{m=1}^N\sum_{i=1}^{N-1}\sum_{j=i+1}^N\int_{B_{\ep}(\v d_m)}\bigg(C(\u_i\cdot\u_j)+K(\w_i\cdot\w_j)+R(\u_i\cdot\w_j)+R(\u_j\cdot\w_i)\bigg)\ud{x}\ud{y}\no\\
&&+\sum_{m=1}^N\int_{B_{\ep}(\v d_m)}\frac{1}{2}\bigg(C\abs{\na u_{0}}^2+K\abs{\na w_{0}}^2+2R(\na u_{0}\cdot\na w_{0})\bigg)\ud{x}\ud{y}\no\\
&&+\sum_{m=1}^N\sum_{i=1}^N\int_{\p B_{\ep}(\v d_m)}\bigg(u_{0}(C\u_i+R\w_i)+w_{0}(K\w_i+R\u_i)\bigg)\cdot n\ud{s},\no
\end{eqnarray*}
with
\begin{eqnarray*}
D_k^VF=D_k^VG+D_k^VH.
\end{eqnarray*}
We divide the proof into several steps:\\
\ \\
Step 1: Estimate of $D_k^VG$.\\
We write
\begin{eqnarray*}
I:&=&D_k^V\bigg(\int_{\Omega_{\ep}}\frac{1}{2}\bigg(C\abs{\u_{0}}^2+K\abs{\w_{0}}^2+2R(\u_{0}\cdot\w_{0})\bigg)\ud{x}\ud{y}\bigg)\no\\
&=&\int_{\Omega_{\ep}}\bigg(C\u_0\cdot D_k^V\u_0+K\w_0\cdot D_k^V\w_0+R\u_0\cdot D_k^V\w_0+R\w_0\cdot D_k^V\u_0\bigg)\ud{x}\ud{y}\no\\
&&-\int_{\p B_{\ep}(\v d_k)}\frac{1}{2}\bigg(C\abs{\u_{0}}^2+K\abs{\w_{0}}^2+2R(\u_{0}\cdot\w_{0})\bigg)\v V\cdot n\ud{s}.\no
\end{eqnarray*}
Hence, by the equations (\ref{temp 1}), we have
\begin{eqnarray*}
&&\int_{\Omega_{\ep}}\bigg(C\u_0\cdot D_k^V\u_0+K\w_0\cdot D_k^V\w_0+R\u_0\cdot D_k^V\w_0+R\w_0\cdot D_k^V\u_0\bigg)\ud{x}\ud{y}\\
&=&\int_{\Omega_{\ep}}\bigg(C\u_0\cdot \na(D_k^Vu_0-\u_k\cdot\v V)+K\w_0\cdot \na(D_k^Vw_0-\w_k\cdot\v V)\no\\
&&+R\u_0\cdot \na(D_k^Vw_0-\w_k\cdot\v V)+R\w_0\cdot \na(D_k^Vu_0-\u_k\cdot\v V)\bigg)\ud{x}\ud{y}\no\\
&=&-\sum_{j=1}^N\int_{\p B_{\ep}(\v d_j)}\bigg(C\u_0\cdot (D_k^Vu_0-\u_k\cdot\v V)\cdot n+K\w_0\cdot (D_k^Vw_0-\w_k\cdot\v V)\cdot n\no\\
&&+R\u_0\cdot (D_k^Vw_0-\w_k\cdot\v V)\cdot n+R\w_0\cdot (D_k^Vu_0-\u_k\cdot\v V)\cdot n\bigg)\ud{s}.\no
\end{eqnarray*}
We obtain
\begin{eqnarray*}
&&D_k^V\bigg(\int_{\Omega_{\ep}}\frac{1}{2}\bigg(C\abs{\u_{0}}^2+K\abs{\w_{0}}^2+2R(\u_{0}\cdot\w_{0})\bigg)\ud{x}\ud{y}\bigg)\\
&=&-\sum_{j=1}^N\int_{\p B_{\ep}(\v d_j)}\bigg(C\u_0\cdot (D_k^Vu_0-\u_k\cdot\v V)\cdot n+K\w_0\cdot (D_k^Vw_0-\w_k\cdot\v V)\cdot n\no\\
&&+R\u_0\cdot (D_k^Vw_0-\w_k\cdot\v V)\cdot n+R\w_0\cdot (D_k^Vu_0-\u_k\cdot\v V)\cdot n\bigg)\ud{s}\no\\
&&-\int_{\p B_{\ep}(\v d_k)}\frac{1}{2}\bigg(C\abs{\u_{0}}^2+K\abs{\w_{0}}^2+2R(\u_{0}\cdot\w_{0})\bigg)\v V\cdot n\ud{s}\no\\
&=&-\int_{\p B_r(\v d_k)}\bigg(\f[\u_0,\w_0]\i-(C\u_0\otimes\u_0+K\w_0\otimes\w_0+R\u_0\otimes\w_0+R\w_0\otimes\u_0)\bigg)\v V\cdot n\ud{s}\no\\
&&-\sum_{j\neq k}\int_{\p B_{\ep}(\v d_j)}\bigg(C\u_0\cdot (D_k^Vu_0-\u_k\cdot\v V)\cdot n+K\w_0\cdot (D_k^Vw_0-\w_k\cdot\v V)\cdot n\no\\
&&+R\u_0\cdot (D_k^Vw_0-\w_k\cdot\v V)\cdot n+R\w_0\cdot (D_k^Vu_0-\u_k\cdot\v V)\cdot n\bigg)\ud{s}\no\\
&&-\int_{\p B_{\ep}(\v d_k)}\bigg(C\u_0\cdot (D_k^Vu_0-\u_k\cdot\v V)\cdot n+K\w_0\cdot (D_k^Vw_0-\w_k\cdot\v V)\cdot n\no\\
&&+R\u_0\cdot (D_k^Vw_0-\w_k\cdot\v V)\cdot n+R\w_0\cdot (D_k^Vu_0-\u_k\cdot\v V)\cdot n\no\\
&&+(C\u_0\otimes\u_0+K\w_0\otimes\w_0+R\u_0\otimes\w_0+R\w_0\otimes\u_0)\v V\cdot n\bigg)\ud{s}\no\\
&=&I_1+I_2+I_3.\no
\end{eqnarray*}
In above estimates, $I_1$ is the desired term, so we only focus on $I_2$ and $I_3$. We need to cancel
\begin{eqnarray*}
I_2&=&-\sum_{j\neq k}\int_{\p B_{\ep}(\v d_j)}\bigg(C\u_0\cdot (D_k^Vu_0-\u_k\cdot\v V)\cdot n+K\w_0\cdot (D_k^Vw_0-\w_k\cdot\v V)\cdot n\no\\
&&+R\u_0\cdot (D_k^Vw_0-\w_k\cdot\v V)\cdot n+R\w_0\cdot (D_k^Vu_0-\u_k\cdot\v V)\cdot n\bigg)\ud{s}\no,
\end{eqnarray*}
and
\begin{eqnarray*}
I_3&=&-\int_{\p B_{\ep}(\v d_k)}\bigg(C\u_0\cdot (D_k^Vu_0-\u_k\cdot\v V)\cdot n+K\w_0\cdot (D_k^Vw_0-\w_k\cdot\v V)\cdot n\no\\
&&+R\u_0\cdot (D_k^Vw_0-\w_k\cdot\v V)\cdot n+R\w_0\cdot (D_k^Vu_0-\u_k\cdot\v V)\cdot n\no\\
&&+(C\u_0\otimes\u_0+K\w_0\otimes\w_0+R\u_0\otimes\w_0+R\w_0\otimes\u_0)\v V\cdot n\bigg)\ud{s}\no\\
&=&-\int_{\p B_{\ep}(\v d_k)}\bigg(C\u_0\cdot (D_{\theta}u_0+\sum_{j\neq k}\u_j\cdot\v V)\cdot n+K\w_0\cdot (D_{\theta}w_0+\sum_{j\neq k}\w_j\cdot\v V)\cdot n\no\\
&&+R\u_0\cdot (D_{\theta}w_0+\sum_{j\neq k}\w_j\cdot\v V)\cdot n+R\w_0\cdot (D_{\theta}u_0+\sum_{j\neq k}\u_j\cdot\v V)\cdot n.\no
\end{eqnarray*}
\ \\
Step 2: Estimate of $D_k^VH$ - First Term.\\
We directly write
\begin{eqnarray*}
II:&=&D_k^V\bigg(\sum_{i=1}^N\sum_{m\neq i}\int_{B_{\ep}(\v d_m)}\frac{1}{2}\bigg(C\abs{\u_{i}}^2+K\abs{\w_{i}}^2+2R(\u_{i}\cdot\w_{i})\bigg)\ud{x}\ud{y}\bigg)\\
&=&D_k^V\bigg(\sum_{m\neq k}\int_{B_{\ep}(\v d_k)}\frac{1}{2}\bigg(C\abs{\u_{m}}^2+K\abs{\w_{m}}^2+2R(\u_{m}\cdot\w_{m})\bigg)\ud{x}\ud{y}\bigg)\no\\
&&+D_k^V\bigg(\sum_{m\neq k}\sum_{i\neq m}\int_{B_{\ep}(\v d_m)}\frac{1}{2}\bigg(C\abs{\u_{i}}^2+K\abs{\w_{i}}^2+2R(\u_{i}\cdot\w_{i})\bigg)\ud{x}\ud{y}\bigg)\no\\
&=&II_1+II_2.\no
\end{eqnarray*}
In $II_1$, we know each $D_{\theta}\u_m=D_{\theta}\w_m=0$ since $m\neq k$, then we have
\begin{eqnarray*}
II_1&=&\sum_{m\neq k}\int_{B_{\ep}(\v d_k)}\bigg(C\u_m\cdot\na(\u_m\cdot\v V)+K\w_m\cdot\na(\w_m\cdot\v V)\\
&&+R\u_m\cdot\na(\w_m\cdot\v V)+R\w_m\cdot\na(\u_m\cdot\v V)\bigg)\ud{x}\ud{y}\no\\
&=&\sum_{m\neq k}\int_{\p B_{\ep}(\v d_k)}\bigg(C\u_m\cdot n(\u_m\cdot\v V)+K\w_m\cdot n(\w_m\cdot\v V)\no\\
&&+R\u_m\cdot n(\w_m\cdot\v V)+R\w_m\cdot n(\u_m\cdot\v V)\bigg)\ud{s}.\no
\end{eqnarray*}
Also, since the domain and functions do not move for $i\neq k$, we have
\begin{eqnarray*}
II_2&=&D_k^V\bigg(\sum_{m\neq k}\int_{B_{\ep}(\v d_m)}\frac{1}{2}\bigg(C\abs{\u_{k}}^2+K\abs{\w_{k}}^2+2R(\u_{k}\cdot\w_{k})\bigg)\ud{x}\ud{y}\bigg)\\
&=&-\sum_{m\neq k}\int_{B_{\ep}(\v d_m)}\bigg(C\u_k\cdot\na(\u_k\cdot\v V)+K\w_k\cdot\na(\w_k\cdot\v V)\no\\
&&+R\u_k\cdot\na(\w_k\cdot\v V)+R\w_k\cdot\na(\u_k\cdot\v V)\bigg)\ud{x}\ud{y}\no\\
&=&-\sum_{m\neq k}\int_{\p B_{\ep}(\v d_m)}\bigg(C\u_k\cdot n(\u_k\cdot\v V)+K\w_k\cdot n(\w_k\cdot\v V)\no\\
&&+R\u_k\cdot n(\w_k\cdot\v V)+R\w_k\cdot n(\u_k\cdot\v V)\bigg)\ud{s}.\no
\end{eqnarray*}
\ \\
Step 3: Estimate of $D_k^VH$ - Second Term.\\
We directly decompose
\begin{eqnarray*}
III:&=&D_k^V\Bigg(\sum_{i<j}\int_{B_{\ep}(\v d_k)}\bigg(C(\u_i\cdot\u_j)+K(\w_i\cdot\w_j)+R(\u_i\cdot\w_j)+R(\u_j\cdot\w_i)\bigg)\ud{x}\ud{y}\Bigg)\\
&&+D_k^V\Bigg(\sum_{m\neq k}\sum_{i<j}\int_{B_{\ep}(\v d_m)}\bigg(C(\u_i\cdot\u_j)+K(\w_i\cdot\w_j)+R(\u_i\cdot\w_j)+R(\u_j\cdot\w_i)\bigg)\ud{x}\ud{y}\Bigg)\no\\
&=&III_1+III_2.\no
\end{eqnarray*}
Then we have
\begin{eqnarray*}
III_1&=&\sum_{i\neq k}\sum_{j\neq i}\int_{B_{\ep}(\v d_k)}\bigg(C\u_i\cdot\na(\u_j\cdot\v V)+K\w_i\cdot\na(\w_j\cdot\v V)\\
&&+R\u_i\cdot\na(\w_j\cdot\v V)+R\w_i\cdot\na(\u_j\cdot\v V)\bigg)\ud{x}\ud{y}\no\\
&=&\sum_{i\neq k}\sum_{j\neq i}\int_{\p B_{\ep}(\v d_k)}\bigg(C\u_i\cdot n(\u_j\cdot\v V)+K\w_i\cdot n(\w_j\cdot\v V)\no\\
&&+R\u_i\cdot n(\w_j\cdot\v V)+R\w_i\cdot n(\u_j\cdot\v V)\bigg)\ud{s}.\no
\end{eqnarray*}
\begin{eqnarray*}
III_2&=&-\sum_{m\neq k}\sum_{i\neq k}\int_{B_{\ep}(\v d_m)}\bigg(C\u_i\cdot\na(\u_k\cdot\v V)+K\w_i\cdot\na(\w_k\cdot\v V)\\
&&+R\u_i\cdot\na(\w_k\cdot\v V)+R\w_i\cdot\na(\u_k\cdot\v V)\bigg)\ud{x}\ud{y}\no\\
&=&-\sum_{m\neq k}\sum_{i\neq k}\int_{\p B_{\ep}(\v d_m)}\bigg(C\u_i\cdot n(\u_k\cdot\v V)+K\w_i\cdot n(\w_k\cdot\v V)\no\\
&&+R\u_i\cdot n(\w_k\cdot\v V)+R\w_i\cdot n(\u_k\cdot\v V)\bigg)\ud{s}.\no
\end{eqnarray*}
\ \\
Step 4: Estimate of $D_k^VH$ - Third Term.\\
We directly decompose
\begin{eqnarray*}
IV:&=&D_k^V\Bigg(\int_{B_{\ep}(\v d_k)}\frac{1}{2}\bigg(C\abs{\na u_{0}}^2+K\abs{\na w_{0}}^2+2R(\na u_{0}\cdot\na w_{0})\bigg)\ud{x}\ud{y}\Bigg)\\
&&+D_k^V\Bigg(\sum_{m\neq k}\int_{B_{\ep}(\v d_m)}\frac{1}{2}\bigg(C\abs{\na u_{0}}^2+K\abs{\na w_{0}}^2+2R(\na u_{0}\cdot\na w_{0})\bigg)\ud{x}\ud{y}\Bigg)\no\\
&=&IV_1+IV_2.\no
\end{eqnarray*}
By integrating by parts, we know
\begin{eqnarray*}
IV_1&=&\int_{\p B_{\ep}(\v d_k)}\bigg(C\na u_0\cdot n(D_k^Vu_0+\na u_0\cdot\v V)+K\na w_0\cdot n(D_k^Vw_0+\na w_0\cdot\v V)\\
&&+R\na u_0\cdot n(D_k^Vw_0+\na w_0\cdot\v V)+R\na w_0\cdot n(D_k^Vu_0+\na u_0\cdot\v V)\bigg)\ud{s}.\no
\end{eqnarray*}
Similarly, we have
\begin{eqnarray*}
IV_2&=&\sum_{m\neq k}\int_{\p B_{\ep}(\v d_m)}\bigg(CD_k^Vu_0(\na u_0\cdot n)+KD_k^Vw_0(\na w_0\cdot n)\\
&&+RD_k^Vw_0(\na u_0\cdot n)+RD_k^Vu_0(\na w_0\cdot n)\bigg)\ud{s}.\no
\end{eqnarray*}
\ \\
Step 5: Estimate of $D_k^VH$ - Fourth Term.\\
We directly decompose
\begin{eqnarray*}
V:&=&D_k^V\Bigg(\sum_{i=1}^N\int_{\p B_{\ep}(\v d_k)}\bigg(u_{0}(C\u_i+R\w_i)+w_{0}(K\w_i+R\u_i)\bigg)\cdot n\ud{s}\Bigg)\\
&&+D_k^V\Bigg(\sum_{m\neq k}\sum_{i=1}^N\int_{\p B_{\ep}(\v d_m)}\bigg(u_{0}(C\u_i+R\w_i)+w_{0}(K\w_i+R\u_i)\bigg)\cdot n\ud{s}\Bigg)\no\\
&=&V_1+V_2.\no
\end{eqnarray*}
Similarly to previous steps, we have
\begin{eqnarray*}
V_1&=&\sum_{i=1}^N\int_{\p B_{\ep}(\v d_k)}\bigg(D_{\theta}u_{0}(C\u_i+R\w_i)+D_{\theta}w_{0}(K\w_i+R\u_i)\bigg)\cdot n\ud{s}\\
&&+\sum_{i\neq k}\int_{\p B_{\ep}(\v d_k)}\bigg(\na u_{0}(C\u_i+R\w_i)\cdot\v V+\na w_{0}(K\w_i+R\u_i)\cdot\v V\bigg)\cdot n\ud{s}.\no
\end{eqnarray*}
Also, we have
\begin{eqnarray*}
V_2&=&\sum_{m\neq k}\sum_{i=1}^N\int_{\p B_{\ep}(\v d_m)}\bigg(D_k^Vu_{0}(C\u_i+R\w_i)+D_k^Vw_{0}(K\w_i+R\u_i)\bigg)\cdot n\ud{s}\\
&&-\sum_{m\neq k}\sum_{i=1}^N\int_{\p B_{\ep}(\v d_k)}\bigg(\na u_{0}(C\u_k+R\w_k)\cdot\v V+\na w_{0}(K\w_k+R\u_k)\cdot\v V\bigg)\cdot n\ud{s}.\no
\end{eqnarray*}
\ \\
Step 6: Synthesis.\\
Collecting all above terms, we have
\begin{eqnarray*}
II_1+III_1+IV_1+V_1&=&\int_{\p B_{\ep}(\v d_k)}\bigg(C\u_0\cdot (D_{\theta}u_0+\sum_{j\neq k}\u_j\cdot\v V)\cdot n+K\w_0\cdot (D_{\theta}w_0+\sum_{j\neq k}\w_j\cdot\v V)\cdot n\no\\
&&+R\u_0\cdot (D_{\theta}w_0+\sum_{j\neq k}\w_j\cdot\v V)\cdot n+R\w_0\cdot (D_{\theta}u_0+\sum_{j\neq k}\u_j\cdot\v V)\cdot n\no\\
&=&-I_3.\no
\end{eqnarray*}
and
\begin{eqnarray*}
II_2+III_2+IV_2+V_2&=&\sum_{j\neq k}\int_{\p B_{\ep}(\v d_j)}\bigg(C\u_0\cdot (D_k^Vu_0-\u_k\cdot\v V)\cdot n+K\w_0\cdot (D_k^Vw_0-\w_k\cdot\v V)\cdot n\no\\
&&+R\u_0\cdot (D_k^Vw_0-\w_k\cdot\v V)\cdot n+R\w_0\cdot (D_k^Vu_0-\u_k\cdot\v V)\cdot n\bigg)\ud{s}\no\\
&=&-I_2.\no
\end{eqnarray*}
Summarizing all above, we obtain
\begin{eqnarray*}
&&I+II+III+IV+V=I_1\\
&=&-\int_{\p B_r(\v d_k)}\bigg(\f[\u_0,\w_0]\i-(C\u_0\otimes\u_0+K\w_0\otimes\w_0+R\u_0\otimes\w_0+R\w_0\otimes\u_0)\bigg)\v V\cdot n\ud{s}.\no
\end{eqnarray*}
Then our result naturally follows.
\end{proof}

\section*{Acknowledgements}

The research is supported by NSF grant
0967140. The author wishes to thank Prof. Irene Fonseca and Prof. Giovanni Leoni for stimulating discussion and valuable comments. Also, the author shows gratitude to the Center for Nonlinear Analysis, PIRE (Partnerships for International Research and Education) project, and Carnegie Mellon University for their support during the preparation of this paper.

\bibliographystyle{siam}
\bibliography{Reference}

\end{document}